\documentclass[oneside,12pt]{amsart}
\usepackage{enumerate,amsfonts,amsmath,amsthm,amssymb,latexsym,verbatim,amscd,mathrsfs}
\usepackage[usenames]{color}
\usepackage{hyperref}

\usepackage[left=2cm,right=1.5cm,top=1.5cm,bottom=1.5cm]{geometry}

\usepackage{fancyhdr}
\numberwithin{equation}{section}

\theoremstyle{definition}
\newtheorem{thm}{Theorem}[section]
\newtheorem{prop}[thm]{Proposition}
\newtheorem{cor}[thm]{Corollary}
\newtheorem{lem}[thm]{Lemma}
\newtheorem{defn}[thm]{Definition}
\newtheorem{rem}{Remark}
\newtheorem{exe}{Example}

\newcommand{\lie}[1]{\mathfrak{#1}}

\def\cal#1{\mathcal{#1}}

\allowdisplaybreaks[1]

\usepackage{enumerate,amsfonts,amsmath,amsthm,amssymb,latexsym,verbatim,amscd,mathrsfs}
\usepackage[usenames]{color}
\usepackage{hyperref}
\theoremstyle{definition}

\def\cal#1{\mathcal{#1}}

\begin{document}

\title[]{Bases for local Weyl modules for the \\ hyper and truncated current $\lie{sl}_2$--algebras}
\author{Angelo Bianchi}
\email{acbianchi@unifesp.br} 
\address{Department of Science and Technology - Federal University of S\~ao Paulo (ICT - Unifesp), Brazil}
\thanks{A.B. is partially supported by CNPq grant 462315/2014-2 and FAPESP grants 2015/22040-0 and 2014/09310-5.}

\author{Evan Wilson}
\email{wilsoneaster@gmail.com} 
\address{Mathematics Department - Franklin \& Marshall College, Lancaster, PA, USA      }
\thanks{E. W. was partially supported by FAPESP grant 2011/12079-5.}

\subjclass[2010]{Mathematics Subject Classification. 17B10, 81R10}
\keywords{Gr\"obner-Shirshov bases, hyper current algebra, local Weyl modules}

\date{}

\dedicatory{}

\begin{abstract}
	We use the theory of Gr\"obner-Shirshov bases for ideals to construct linear bases for graded local Weyl modules for the (hyper) current and the truncated current algebras associated to the finite-dimensional complex simple Lie algebra $\mathfrak{sl}_2$. The main result is a characteristic-free construction of bases for this important family of modules for the hyper current $\mathfrak{sl}_2$-algebra. In the positive characteristic setting this work represents the first construction in the literature. In the characteristic zero setting, the method provides a different construction of the Chari-Pressley (also Kus-Littelmann) bases and the Chari-Venkatesh bases for local Weyl modules for the current $\mathfrak{sl}_2$-algebra. Our construction allows us to obtain new bases for the local Weyl modules for truncated current $\mathfrak{sl}_2$-algebras with very particular properties.
\end{abstract}

\maketitle

\section*{Introduction}

The category of level zero representations of affine (and quantum affine) algebras and its full subcategory of finite-dimensional representations have been extensively studied in the last three decades. The representation theory of these algebras gives a significant contribution to identifying interesting families of finite-dimensional representations of \textit{loop} and \textit{current algebras}, such as the universal finite-dimensional highest weight modules called \textit{local Weyl modules}, which became objects of independent and deep interest in the last ten years (cf. \cite{Csurvey,CFK,CFS,FNS,JMsurvey,nehsav,nehsav2,nehsavsen}). Among the generalizations of current algebras (as in \cite{CFK}) special attention is given to the \textit{truncated current algebras} (cf. \cite{GF,kus,bw}), which are finite-dimensional quotients of the current algebra.
			
During the last decade, the study of the positive characteristic analogues of these local Weyl modules for current and loop algebras was also developed in \cite{bm,BMM,hyperlar}, where we can see that the characteristic zero and positive characteristic cases share several properties. To differentiate the positive characteristic case for the current algebra, we will refer to it as \textit{hyper current algebra}.

Even with a large number of papers dedicated to the study of structure, character, decomposition, tensor product, fusion product, and reducibility of Weyl modules, only few are devoted to constructing bases for these modules (cf. \cite{CL,CPweyl,CV,kus}). Additionally, some recent papers (cf. \cite{stab,stab2,stab1}) focused in studying properties of the bases constructed in \cite{CL,CPweyl}. 

A basis for local Weyl modules for the current $\lie{sl}_2$-algebra was first constructed by Chari-Pressley \cite{CPweyl} and it was used by Chari-Loktev \cite{CL} in the construction of a basis for local Weyl modules for current algebras associated to $\lie{sl}_{n+1}, \ n> 1$. Recently, two constructions came up: Kus-Littelmann \cite{kus} constructed a basis for truncated local Weyl modules for the current $\lie{sl}_{2}$-algebra whose construction contemplates the Chari-Pressley basis for graded local Weyl modules by using a very different approach; and Chari-Venkatesh \cite{CV} provided the construction of a different basis for graded local Weyl modules for the current $\lie{sl}_{2}$-algebra. Unfortunately, we still do not see how to adapt any of those constructions to the positive characteristic setting. This fact already appeared with many other results first proved in characteristic zero and then generalized to positive characteristic setting by using very different tools as in the pairs of papers \cite{CPweyl,hyperlar} and \cite{bm,CFS}. 

The present paper was originally intended to present a different approach to obtain some new bases for graded local Weyl modules for the current $\lie{sl}_2$-algebra. While this work was developed, we noticed that we could provide a characteristic-free construction of a basis for graded local Weyl modules for the current and hyper current $\lie{sl}_2$-algebra. This work provides the first explicit construction of bases for hyper current algebras (see Corollary \ref{cor:basis}) and it is also an alternative construction to get the Chari-Pressley (cf. \cite{CPweyl}) and Kus-Littelmann (cf. \cite{kus}) bases in the characteristic zero setting. Further, our method is also an alternative construction (see Theorem \ref{t:revbasis}) to get Chari-Venkatesh basis (cf. \cite{CV}) with some advantages due its interpretation coming from monomial orders. Our method made possible to construct a distinguished basis for truncated local Weyl modules for the truncated current $\lie{sl}_2$-algebra (see Theorem \ref{truncbasis}).

The paper is organized as follows: in Section \ref{pre} we define the basic objects and introduce some tools and results; in Section \ref{genGarland} we prove a generalized version of Garland's formula; in Section \ref{gsb} we present the construction of Gr\"obner-Shirshov bases for polynomial rings and divided polynomial algebras; in Section \ref{mods} we provide the definition of local Weyl modules and basic results in a general setting; in Section \ref{s:sl2} we specialize to the case $\lie{sl}_2$ and construct a basis for graded local Weyl modules; in Section \ref{app} we give some applications of the main result in the previous section; Section \ref{proofs} provides the proofs of the two main results of the paper.

\

\noindent \textbf{Acknowledgment:} We wish to thank Dr. Vyacheslav Futorny for support, Dr. Adriano Moura for helpful suggestions and discussions, and Dr. Ghislain Fourier for calling our attention to the truncated case.

\section{Preliminaries}\label{pre}

In this section we fix some notation and review basic facts about the base algebras. We always denote by $\mathbb C$, $\mathbb Z$, $\mathbb Z_+$ and $\mathbb N$ the sets of complex, integer, nonnegative integer and natural numbers, respectively.

\subsection{Simple Lie algebras and their current algebras} 

Let $\lie g$ be a finite-dimensional complex simple Lie algebra with a fixed Cartan subalgebra $\lie h$. Denote by $R$ the associated root system, fix a choice of positive roots $R^+\subseteq R$ and let $\lie g=\lie n^-\oplus\lie h\oplus \lie n^+$ be the corresponding triangular decomposition. The associated simple roots and fundamental weights will be denoted respectively by $\alpha_i$ and $\omega_i$, $i\in I$, where $I$ is an indexing set for the nodes of the Dynkin diagram of $\lie g$. Denote by $P$ and $Q$ the weight and root lattices of $\lie g$. Let $P^+$ and $Q^+$ be the $\mathbb Z_+$-span of the fundamental weights and simple roots of $\lie g$, respectively. For notational convenience, fix a Chevalley basis $\{x_\alpha^\pm, h_i:\alpha\in R^+, i\in I\}$ of $\lie g$.

Let $\lie g[t] = \lie g\otimes \mathbb C[t]$ denote the corresponding current algebra (namely, the extension by scalars of $\lie g$ to the polynomial ring $\mathbb C[t]$),  where the bracket is given by $[x\otimes f, y\otimes g] =[x,y]\otimes fg$ for $x,y\in\lie g$ and $f,g\in\mathbb C[t]$. Notice that $\lie g\otimes 1$ is a subalgebra of $\lie g[t]$ isomorphic to $\lie g$ and, if $\lie b$ is a subalgebra of $\lie g$, then $\lie b[t]=\lie b\otimes \mathbb C[t]$ is naturally a subalgebra of $\lie g[t]$. In particular, we have $\lie g[t] = \lie n^-[t] \oplus \lie h[t] \oplus \lie n^+[t]$ and $\lie h[t]$ is an abelian subalgebra of $\lie g[t]$. 

For a Lie algebra $\lie a$, we denote by $U(\lie a)$ the corresponding universal enveloping algebra of the Lie algebra $\lie a$. The PBW theorem implies that the multiplication establishes isomorphisms $U(\lie g[t])\cong U(\lie n^-[t])\otimes U(\lie h[t])\otimes U(\lie n^+[t]).$

\subsection{Integral forms and hyperalgebras}\label{s:hyper}

For a fixed order on the Chevalley basis of $\lie g[t]$ and a PBW monomial with respect to this order, we construct an ordered monomial in the elements of the set
\begin{equation*}
{\cal M}[t]=\left\{(x^\pm_{\alpha}\otimes t^r)^{(k)}, \Lambda_{\alpha_i, k},\tbinom{h_{i}}{k}  :  \alpha \in R^+, i\in I, k, r\in \mathbb Z_+\right\}  
\end{equation*}
where
$$(x^\pm_{\alpha}\otimes t^r)^{(k)}:=\frac{(x^\pm_{\alpha}\otimes t^r)^{k}}{k!},\qquad  \binom{h_i}{k}:=\frac{h_i(h_i-1)\dots(h_i-k+1)}{k!}$$
and 
$$\Lambda_{\alpha}(u) := \sum_{r=0}^\infty \Lambda_{\alpha,r}u^r := \exp \left(-\sum_{s=1}^{\infty}\frac{h_{\alpha}\otimes t^{s}}{s} u^{s}\right),$$
by using the correspondence $(x^\pm_\alpha\otimes t^r)^k \leftrightarrow (x^\pm_\alpha \otimes t^r)^{(k)}$, $h_{i}^k \leftrightarrow \tbinom{h_{i}}{k}$ and $(h_{i}\otimes t^r)^k \leftrightarrow (\Lambda_{\alpha_i,r})^k$.
Using a similar correspondence we also consider monomials in $U(\lie g)$ formed by elements of
$$\cal M=\left\{ (x_{\alpha}^\pm)^{(k)},\tbinom{h_i}{k}  :  \alpha\in R^+, i\in I, k\in\mathbb Z_+\right\}.$$
Notice that $\cal M$ can be naturally regarded as a subset of ${\cal M}[t]$.
The set of ordered monomials thus obtained are bases of $U(\lie g[t])$ and $U(\lie g)$, respectively.

Let $U_{\mathbb Z}(\lie g[t]) \subseteq U(\lie g[t])$ and $U_{\mathbb Z}(\lie g) \subseteq U(\lie g)$ be the $\mathbb Z$--subalgebras generated respectively by $\{(x^\pm_{\alpha}\otimes t^r)^{(k)} :  \alpha \in R^+, \ r,k\in\mathbb Z_+\}$, 
and $\{(x_{\alpha}^\pm)^{(k)}  :  \alpha\in R^+, k\in\mathbb Z_+\}$. The following crucial theorem was proved in \cite{kosagz}, in the $U(\lie g)$ case, and in \cite{G,mitz}  (see also \cite{sam}) for the $U(\lie g[t])$ case.

\begin{thm} \label{forms}
	The subalgebras $U_{\mathbb Z}(\lie g)$ and $U_{\mathbb Z}(\lie g[t])$ are free $\mathbb Z$-modules and the sets of ordered monomials constructed from $\cal M$ and ${\cal M}[t]$ are $\mathbb Z$-bases of $U_{\mathbb Z}(\lie g)$ and $U_{\mathbb Z}(\lie g[t])$, respectively. \hfill\qedsymbol
\end{thm}
In other words, $U_\mathbb Z(\lie g[t])$ is an integral form of $U(\lie g[t])$ (similarly for $U_\mathbb Z(\lie g)$).

\

	If $\lie a$ is any subalgebra of $\lie g$, set $U_\mathbb Z(\lie a[t]):= U(\lie a[t] )\cap U_\mathbb Z(\lie g[t])$ and  $U_\mathbb Z(\lie a): = U(\lie a )\cap U_\mathbb Z(\lie g)$. Then, 
\begin{equation*}
\lie a\in\{ \lie g, \lie n^\pm[t],\lie h[t], \lie n^\pm,\lie h\} \quad\Rightarrow\quad \mathbb C\otimes_\mathbb Z U_\mathbb Z(\lie a)\stackrel{\cong}{\longrightarrow} U(\lie a).
\end{equation*}

Given a field  $\mathbb F$, the $\mathbb F$--\textit{hyperalgebra of $\lie a$} is defined by
\begin{equation*}
	U_\mathbb F(\lie a) :=  \mathbb F\otimes_{\mathbb Z}U_\mathbb Z(\lie a).
\end{equation*}
We will refer to $U_\mathbb F(\lie g[t])$  as the \textit{hyper current algebra of $\lie g$ over $\mathbb F$}. 

Notice that we have
\begin{equation}\label{PBWordering}
U_\mathbb F(\lie g)=U_\mathbb F(\lie n^-)U_\mathbb F(\lie h)U_\mathbb F(\lie n^+) \quad\text{and} \quad U_\mathbb F(\lie g[t])=U_\mathbb F(\lie n^-[t])U_\mathbb F(\lie h[t])U_\mathbb F(\lie n^+[t]).
\end{equation}

\begin{rem}\label{hyperrmk} More discussion about hyperalgebras can be found it \cite{bm,BMM,hyperlar}. We recall some important facts:
	\begin{enumerate}
		\item \label{hyperrmk1} If the characteristic of $\mathbb F$ is zero, the algebra $U_\mathbb F(\lie g[t])$ is naturally isomorphic to $U(\lie g[t]_\mathbb F)$ where $\lie g[t]_\mathbb F =  \mathbb F\otimes_\mathbb Z\lie g[t]_\mathbb Z $ and $\lie g[t]_\mathbb Z$ is the $\mathbb Z$-span of the Chevalley basis of $\lie g[t] $.  If the characteristic of $\mathbb F$ is positive, we only have an algebra homomorphism $U(\lie a_\mathbb F)\to U_\mathbb F(\lie a)$ which is neither injective nor surjective. 
		\item Notice that the Hopf algebra structure of the universal enveloping algebras induce such structure on the hyperalgebras. For any Hopf algebra $H$, we denote by $H^0$ its augmentation ideal.
		\item The integral form of $\lie g$ (cf. \cite{kosagz}) coincides with its intersection with the integral form of $U(\lie g[t])$ (cf. \cite{G,mitz}) , i.e $U_\mathbb Z(\lie g)=U(\lie g)\cap U_\mathbb Z(\lie g[t])$, which allows us to regard $U_\mathbb Z(\lie g)$ as a $\mathbb Z$-subalgebra of $U_\mathbb Z(\lie g[t])$. 
	\end{enumerate}	
\end{rem}

\subsection{Partitions}\label{partitions} The language of partitions will be freely used in the text in a natural way associated with exponents of polynomials. A partition of a positive integer number is a sequence $\lambda=(\lambda_1,\lambda_2,\dots, \lambda_k)$ of non increasing positive integers ($\lambda_i\geq \lambda_{i+1}$ and $\lambda_i>0$ for $1\leq i\leq k$). We say $\lambda$ is a partition of $r$, symbolically  $\lambda \vdash r$, if $\sum_{i=1}^k \lambda_i=r$. We also say $|\lambda|=r$ is the sum of the partition and $\ell(\lambda)=k$ is its length. A useful way of graphically representing partitions is the Ferrers diagram which associates a left justified row of $\lambda_i$ boxes for each $i\in\{1,2,\dots, k\}$, decreasing from top to bottom. The transpose of $\lambda$, denoted $\lambda'$ is the partition given by changing rows and columns in $\lambda$. The partial order $\preceq$ known as \emph{dominance order} is defined by $\mu \preceq \lambda$ if, and only if, $\sum_{i=1}^j \mu_i\leq \sum_{i=1}^j \lambda_i$, $1\leq j\leq \ell(\lambda)$.

\section{Some commutation formulas: a Generalization of Garland's formula.} \label{genGarland}

We state some useful commutation identities in hyperalgebras. First we recall a famous formula due to Garland and then we present a generalization of it.

Given $\alpha\in R^+$, consider the following power series with coefficients in $U_\mathbb Z(\lie n^-[t])$:
$$
	X_{\alpha;(m,n)}(u) := \sum_{j=0}^\infty  (x^-_{\alpha}\otimes t^{(m+n)j+m}) u^{j+1}.
$$

In what follows we denote $\lie h\otimes t\mathbb C[t]$ by $t\lie h[t]$ for shortness.

\begin{lem} \cite{mitz} \label{basicrel} Let $\alpha\in R^+$ and  $\ell,k,m,n\in \mathbb Z_+$. We have
	$$(x_\alpha^+\otimes t^n)^{(\ell )}(x_\alpha^-\otimes t^m)^{(k)} = ({X_{\alpha;(m,n)}(u)}^{(k-\ell)})_{k} \mod U_\mathbb Z (\lie g [t]\left(U_\mathbb Z(\lie n^+[t])^0+U_\mathbb Z(t\lie h[t])^0\right)$$ 
	where the index $k$ means the coefficient of $u^k$ of the above power series.
	\quad\hfill\qedsymbol
\end{lem}

We now extend this result. Let $w,u_1,u_2,\dots$  be algebraically independent variables which will be used as formal variables. For any ring $R$, the formal power series ring in the variables $u_1,u_2,\dots$ will be denoted $R[[\mathbf u]]$. Similarly, the formal power series ring in the variables $w,u_1,u_2\dots$ will be denoted $R[[\mathbf u;w]]$. Set $S:=\mathbb C[[\mathbf{u};w]]$

Let $\alpha\in R^+$. We define $$X_\alpha(u):=\sum_{n=0}^{\infty} (x_\alpha^-\otimes t^{n})u^{n+1}=X_{\alpha;(0,1)}(u)\quad \text{ and } \quad H_\alpha(u):=\sum_{n=0}^{\infty} (h_\alpha\otimes t^{n})u^{n+1}.$$ In $U_\mathbb{Z}(\lie{g}[t])[[\mathbf{u}]]$ we define the formal series $Y_\alpha[s]$, for $s\geq 0$, by
$$
	Y_\alpha[s]=\text{Res}_w(G[s](w)X_\alpha(w^{-1}))
$$
where
\begin{equation*}
G[s](w)=\begin{cases}\frac{ 1}{1+u_1w+u_2w^2\cdots +u_sw^s},\text{ if }s\geq 1\\
1, \text{ if }s=0
\end{cases}
\end{equation*}
and $\text{Res}_w(f)$ denotes the coefficient of $w^{-1}$ in $f \in S$ (that is the residue at 0 with respect to the formal variable $w$) understood as a formal series in $w$ with coefficients in the polynomial ring $\mathbb C[u_1,\dots,u_s]$. In particular, if $s=0$ we have
$$
Y_\alpha[s]=\text{Res}_w(X_\alpha(w^{-1}))=\text{Res}_w\left (\sum_{i=0}^\infty (x_\alpha^-\otimes t^i)w^{-i-1}\right )=x_\alpha^-.
$$

Let us record some useful results in formal series. The first is an analogue of a well-known result from vertex operators for loop algebras. Let $\pi_{w}^-$ denote projection onto the negative powers in $w$. We have:
\begin{align*}
\left [ H_\alpha(w^{-1}),X_\alpha(z^{-1})\right ]&=\left [ \sum_{i=0}^\infty (h_\alpha\otimes t^i)w^{-i-1},\sum_{j=0}^\infty (x^-_\alpha\otimes t^j)z^{-j-1} \right]\\
&= \sum_{i=0}^\infty \sum_{j=0}^\infty\left [ h_\alpha\otimes t^i,x^-_\alpha\otimes t^j\right]w^{-i-1}z^{-j-1} \\
&= -2\sum_{i=0}^\infty \sum_{j=0}^\infty (x^-_\alpha\otimes t^{i+j}) w^{-i-1}z^{-j-1} \\
&= -2\sum_{k=0}^\infty \sum_{j=0}^k (x_\alpha^-\otimes t^{k})w^{-(k-j)-1}z^{-j-1} \\
&= -2\pi^-_w\left (\sum_{k=0}^\infty \sum_{j=0}^\infty (x_\alpha^-\otimes t^{k})w^{-(k-j)-1}z^{-j-1}\right )\\
&= -2\pi^-_w\left (\sum_{k=0}^\infty \left ( (x_\alpha^-\otimes t^{k})w^{-k-1}\sum_{j=0}^\infty w^{j}z^{-j-1} \right )\right )\\
&= \pi^-_w\left (X_\alpha(w^{-1})\frac{-2}{z-w} \right )
\end{align*}
with $\frac{1}{z-w}$ being considered a formal series in $w$.

Observe that, for $s\geq 1$ we have:
\begin{equation} \label{eq:Grecur}
G[s](w)=\frac{G[s-1](w)}{1+u_{s}w^{s}G[s-1](w)}. 
\end{equation}
from which we see that
\begin{equation} \label{eq:Gderiv}
\partial_{u_{s}}G[s](w)=-w^sG[s](w)^2.
\end{equation}

Let $f(u_1,\dots,u_s;w),g(u_1,\dots,u_s;w)\in S$. In the following, we suppress the notational dependence on $u_1,\dots,u_s$. We have: 
\begin{align}
	[\text{Res}_w(f(w)H_\alpha(w^{-1})),\text{Res}_w(g(w)X_\alpha(w^{-1}))]&=[\text{Res}_w(f(w)H_\alpha(w^{-1})),\text{Res}_z(g(z)X_\alpha(z^{-1}))]\label{eq:resbracket}\\
	&=\text{Res}_w(f(w)\text{Res}_z(g(z)[H_\alpha(w^{-1}),X_\alpha(z^{-1})]))\nonumber\\
	&=\text{Res}_w\left (f(w)\text{Res}_z\left(g(z)\pi_{w}^-\left (\frac{-2}{z-w}X_\alpha(w^{-1})\right )\right) \right )\nonumber\\
	&=-2\text{Res}_w\left (f(w)\pi_{w}^-\left (X_\alpha(w^{-1})\text{Res}_z\left(\frac{g(z)}{z-w}\right )\right) \right )\nonumber\\
	&=-2\text{Res}_w\left (f(w)\pi_w^-(g(w)X_\alpha(w^{-1})) \right )\nonumber\\
	&=-2\text{Res}_w\left (f(w)g(w)X_\alpha(w^{-1}) \right )\nonumber
\end{align}
where the first equation is a change of variables, in the second we used the linearity properties of the bracket and residue, in the third we used the previous result, and in the fifth we used a standard property of the residue. In the sixth equation, $\pi_w^-$ disappears in the residue since $f(w)$ contains only positive powers of $w$.

The main result we are going to prove in this subsection is:

\begin{prop}\label{prop:commutation}
	Let $\alpha\in R^+$, $k\geq 0$, and  $a_1,\dots,a_s\in \mathbb Z_{\geq 0}, s\ge 1$ such that $\sum_{i=1}^s a_i\leq k$. We have that
	\begin{gather*}(x_\alpha^+ \otimes t^{s})^{(a_s)} \cdots (x_\alpha^+ \otimes t^2)^{(a_2)}(x_\alpha^+ \otimes t)^{(a_1)} (x_\alpha^-)^{( k )} = \text{coeff}_{a_1,a_2,\dots,a_s} Y[s]^{(k-a_1-a_2-\cdots-a_{s})} \\ \hspace{9cm}\mod{U_\mathbb{Z}(\lie{g}[t])\left(U_\mathbb{Z}(t\mathfrak{h}[t])^0 + U_\mathbb{Z}(\mathfrak{n}^+[t])^0\right)}
	\end{gather*}
	where $\text{coeff}_{a_1,a_2,\dots,a_s} Y_\alpha[s]^{(k-a_1-a_2-\cdots-a_{s})}$ denotes the coefficient of $u_1^{a_1}u_2^{a_2}\cdots u_s^{a_{s}}$ in the respective series.
\end{prop}

It will be a direct consequence of the following lemma:

\begin{lem} Let $\alpha\in R^+$.
	
	\begin{enumerate}
		\item In $U_\mathbb{Z}(\lie{g}[t])[[{\bf u}]]$ we have, for $r\geq 1,s\geq 1$: 
$$
	[x_\alpha^+\otimes t^{s}, Y_\alpha[s]^{(r)}]=\partial_{u_s}Y_\alpha[s]^{(r-1)}
	\quad \mod{U_\mathbb{Z}(\lie{g}[t])[[{\bf u}]]\left(U_\mathbb{Z}(t\mathfrak{h}[t])^0[[{\bf u}]] +  U_\mathbb{Z}(\mathfrak{n}^+[t])^0[[{\bf u}]]\right)}.
$$

		\item In $U_\mathbb{Z}(\lie{g}[t])[[{\bf u}]]$ we have, for $r\geq 1, 0\leq n\leq r,s\geq 1$:
$$
			(x_\alpha^+\otimes t^{s})^{(n)} Y_\alpha[s-1]^{(r)}=\left(Y_\alpha[s]^{(r-n)}\right )_{n} \quad \mod{U_\mathbb{Z}(\lie{g}[t])[[{\bf u}]]\left(U_\mathbb{Z}(t\mathfrak{h}[t])^0[[{\bf u}]] + U_\mathbb{Z}(\mathfrak{n}^+[t])^0[[{\bf u}]]\right)}
$$
		where the subscript $n$ means the coefficient of $u_s^n$.
		
		\item For all $s\geq 1$ we have 
$$
		Y_\alpha[s]=\sum_{\eta \ : \ \ell(\eta)\leq s} (-1)^{\ell( \eta )}\frac{\ell(\eta)!}{\prod_{i=1}^sm_i(\eta)!}(x_\alpha^-\otimes t^{|\eta|})\prod_{i=1}^s u_i^{m_i(\eta)}.
$$
	\end{enumerate}
\end{lem}

\begin{proof} 
	
 Throughout this proof, let $\mathcal{I}=U_\mathbb{Z}(\lie{g}[t])[[{\bf u}]]\left(U_\mathbb{Z}(t\mathfrak{h}[t])^0[[{\bf u}]] +  U_\mathbb{Z}(\mathfrak{n}^+[t])^0[[{\bf u}]]\right)$.
	\begin{enumerate}	
	\item We use induction on $r$. When $r=1$ we have
		\begin{equation*}
			[x_\alpha^+\otimes t^{s}, Y_\alpha[s]]=\left [x_\alpha^+\otimes t^{s},\text{Res}_w(G[s](w)X_\alpha(w^{-1}))\right ]=0 \mod \mathcal I.
		\end{equation*}
		
		We now assume that statement (1) of the lemma is true for $r$ and compute:
		\begin{align}
			[x_\alpha^+\otimes t^{s}, (r+1)Y_\alpha[s]^{(r+1)}]&=[x_\alpha^+\otimes t^{s}, Y_\alpha[s]Y_\alpha[s]^{(r)}]\label{eq:Ybracket}\\
			&=[x_\alpha^+\otimes t^{s},Y_\alpha[s]]Y_\alpha[s]^{(r)}+Y_\alpha[s][x_\alpha^+\otimes t^{s},Y_\alpha[s]^{(r)}]\nonumber
		\end{align}
		In the first term above we have
		\begin{align*}
			[x_\alpha^+\otimes t^{s},Y_\alpha[s]]Y_\alpha[s]^{(r)}&=[x_\alpha^+\otimes t^s,\text{Res}_w (G[s](w)X_\alpha(w^{-1}))] Y_\alpha[s]^{(r)}\\
			& =\text{Res}_w (w^sG[s](w)H_\alpha(w^{-1})) Y_\alpha[s]^{(r)}\\
			&=Y_\alpha[s]^{(r-1)}\left [\text{Res}_w (w^sG[s](w)H_\alpha(w^{-1})) ,Y_\alpha[s]\right ]\mod{\mathcal{I}}\\
			&=Y_\alpha[s]^{(r-1)}\text{Res}_w (-2w^sG[s](w)^2X_\alpha(w^{-1})) \mod{\mathcal{I}}\\
			&=Y_\alpha[s]^{(r-1)}\text{Res}_w (2(\partial_{u_{s}}G[s](w))X_\alpha(w^{-1})) \mod{\mathcal{I}}\\
			&=2Y_\alpha[s]^{(r-1)}\partial_{u_s}Y_\alpha[s]\mod{\mathcal{I}},
		\end{align*}
		where in the third equation we used induction on $r$, in the fourth we used (\ref{eq:resbracket}) and in the fifth we used (\ref{eq:Gderiv}). 
		
		In the second term of (\ref{eq:Ybracket}) we have:
		\begin{align*}
			Y_\alpha[s][x_\alpha^+\otimes t^{s},Y_\alpha[s]^{(r)}]&=Y_\alpha[s]\partial_{u_s}(Y_\alpha[s]^{(r-1)})\mod{\mathcal I}\\
			&=Y_\alpha[s]Y_\alpha[s]^{(r-2)}\partial_{u_s}(Y_\alpha[s])\mod{\mathcal I}\\
			&=(r-1)Y_\alpha[s]^{(r-1)}\partial_{u_s}Y_\alpha[s]\mod{\mathcal{I}}.
		\end{align*}
		Combining terms, we get
		$$
		(r+1)[x_\alpha^+\otimes t^{s}, Y_\alpha[s]^{(r+1)}]=(r+1)\partial_{u_s}Y_\alpha[s]^{(r)}\mod{\mathcal{I}}.
		$$
		Since $r+1\neq 0$ we see that $[x_\alpha^+\otimes t^{s}, Y_\alpha[s]^{(r+1)}]-\partial_{u_s}Y_\alpha[s]^{(r)} \in \mathcal{I}$.

		\item We use induction on $n$. For the $n=0$ case, we want to show
		$$
		Y_\alpha[s-1]^{(r)}=\left(Y_\alpha[s]^{(r)}\right )_{0}.
		$$

		We compute:
		\begin{align*}
		(Y_\alpha[s]^{(r)})_0&=\left.Y_\alpha[s]^{(r)}\right|_{u_s=0}\\
		&=\left.\text{Res}_w(G[s](w))X_\alpha(w^{-1}))^{(r)}\right|_{u_s=0}\\
		&=\left.\text{Res}_w\left (\frac{G[s-1](w)}{1+u_{s}w^{s}G[s-1](w)}X_\alpha(w^{-1})\right)^{(r)}\right|_{u_s=0}\\
		&=\text{Res}_w(G[s-1](w)X_\alpha(w^{-1}))^{(r)}\\
		&=Y_\alpha[s-1]^{(r)}
		\end{align*}
		using (\ref{eq:Grecur}) in the third equation.
		
		Now assume the statement is true for all $n<r$. We compute:
		\begin{align*}
			(n+1)(x_\alpha^+\otimes t^{s})^{(n+1)}Y_\alpha[s-1]^{(r)}&=(x_\alpha^+\otimes t^{s}) (x_\alpha^+\otimes t^{s})^{(n)}Y_\alpha[s-1]^{(r)}\\
			&=(x_\alpha^+\otimes t^{s}) \left(Y_\alpha[s]^{(r-n)}\right )_{n}\mod{\mathcal{I}}\\
			&=\left((x_\alpha^+\otimes t^{s}) Y_\alpha[s]^{(r-n)}\right )_{n}\mod{\mathcal{I}}\\
			&=(\partial_{u_{s}}Y_\alpha[s]^{(r-n-1)})_n\mod{\mathcal{I}}\\
			&=(n+1)(Y_\alpha[s]^{(r-n-1)})_{n+1}\mod{\mathcal{I}},
		\end{align*}
		where in the fourth equation, we use the result of (1) and in the fifth we use the ``power rule'' property of the formal derivative. Therefore, $(x_\alpha^+\otimes t^{s})^{(n+1)}Y_\alpha[s-1]^{(r)}-(Y_\alpha[s]^{(r-n-1)})_{n+1}\in \mathcal I$.

		\item We compute:
		\begin{align*}
			G[s](w)&=\frac{1}{1+u_1w+u_2w^2\cdots +u_sw^s}\\
			&=\sum_{n=0}^\infty (-1)^{n} (u_1w+u_2w^2\cdots +u_sw^s)^n\\
			&=\sum_{n=0}^\infty (-1)^{n}\sum_{\nu_1+\cdots\nu_s=n} \frac{n!}{\prod_{i=1}^{s}\nu_{i}!}\prod_{j=1}^s (u_jw^j)^{\nu_j}\\
			&=\sum_{\eta} (-1)^{\ell( \eta )}\frac{\ell(\eta)!}{\prod_{i=1}^sm_i(\eta)!}w^{|\eta|}\prod_{i=1}^s u_i^{m_i(\eta)}
		\end{align*}
		using multinomial expansion in the third equality, and where the sum is over all partitions $\eta$ with $\ell(\eta)\leq s$. The result then follows.\qedhere
	\end{enumerate}
\end{proof}

\section{Polynomials, divided power algebra, and Gr\"obner-Shirshov bases} \label{gsb}

In the Subsection \ref{s:hyper} we introduced the hyperalgebras. In order to get the main result of this paper, we will need to pass to the setting of polynomial rings and the divided power polynomial algebra, which we introduce now. Further, we recall the definition and main properties of Gr\"obner bases for polynomial rings and Gr\"obner-Shirshov bases for divided power polynomial algebra. For additional details we refer to \cite{AL,KL,KLLO}. 

\subsection{Polynomial rings and Gr\"obner bases}\label{Gbases}

Given a field $\mathbb K$ and $n\in \mathbb N$ we denote by $\mathbb K_n=\mathbb K[x_0,\dots,x_{n-1}]$ the polynomial ring in $n$ variables $x_0,\dots,x_{n-1}$. For any set of elements $S\subset \mathbb K_n$ we shall denote by $\langle S \rangle$ the ideal generated by $S$ in $\mathbb K_n$. Any element of the form $x_0^{a_0}\dots x_{n-1}^{a_{n-1}}$, with $(a_0,\dots,a_{n-1})\in \mathbb Z_+^n$, is called a monomial in $\mathbb K_n$.  Let $\prec$ be a monomial order in $\mathbb K_n$. Given $f \in \mathbb K_n$, we write $f=\sum_{i=0}^{k}c_im_i$, where $m_i$ are monomials in $\mathbb K_n$ and $c_i \in \mathbb K$ for each $i\in \{1,\dots,n\}$. We denote by $LM(f)$ the leading monomial of $f$ with respect to $\prec$, and, for any $S\subset \mathbb K_n$, we denote the ideal of leading terms of $S$ by $LM(S) = \langle LM(s) \mid s \in S \rangle$. Given an ideal $I \subset \mathbb K_n$, a set $G=\{ g_1,\dots,g_s \} \subset I$ is called a \textit{ Gr\"obner basis of $I$} if $LM(G)=LM(I)$. A Gr\"obner basis $G=\{ g_1,\dots,g_s \}$ is called a \textit{reduced Gr\"obner basis} if, for all $i\in \{1,\dots,s\}$, $g_i$ is monic and no nonzero term in $g_i$ is divisible by any $LM(g_j)$ for any $j\in \{1,\dots,s\}$, $j\ne i$. 

Similarly, we will denote by $\mathbb K_\infty$ the polynomial ring in infinitely many variables $x_i$, $i\in\mathbb Z_+$, and we naturally consider the same definitions of the previous paragraph.

We state two classical results which we shall use in this work. 

\begin{thm}[\cite{AL}] \label{buch} Fix an monomial order in $\mathbb K_n$. Let $I$ be a nonzero ideal in $\mathbb K_n$. 
	\begin{enumerate}[(a)]
		\item \label{buch1}  The ideal $I$ has a unique (finite) reduced Gr\"obner basis.
		\item \label{buch2}  The set of monomials in $\mathbb K_n$ which are not divisible by any of the leading terms of a Gr\"obner basis for $I$ forms a $\mathbb K$-basis of $\mathbb K_n/I$. \hfill\qedsymbol
	\end{enumerate}
\end{thm}

Given a polynomial ring $\mathbb K[x_1,\dots,x_n,y_1,\dots,y_m]=:\mathbb K[\{X\},\{Y\}]$ in which the variables are split into two subsets $X=\{x_1,\dots,x_n\}$ and $Y=\{y_1,\dots,y_m\}$, a monomial ordering ``eliminating the variables'' $\{x_1,\dots,x_n\}$ is a monomial ordering for which two monomials are compared by first comparing the variables $\{x_1,\dots,x_n\}$, and, in case of equality only, considering the variables $\{y_1,\dots,y_m\}$. This implies that a monomial containing a variable from $X$ is greater than every monomial independent of the variables from $X$. The next theorem simplifies some computations in $G\cap \mathbb K[y_1,\dots,y_m]$:

\begin{thm}\cite[Elimination Theorem]{AL}   \label{elim} If $G$ is a Gr\"obner basis of an ideal $I\subseteq \mathbb K[\{X\},\{Y\}]$ for an elimination monomial ordering, then $G\cap \mathbb K[y_1,\dots,y_m]$ is a Gr\"obner basis of the  \textit{elimination ideal} $I\cap \mathbb K[y_1,\dots,y_m]$. Moreover, a polynomial belongs to $G\cap \mathbb K[y_1,\dots,y_m]$ if, and only if, its leading term belongs to $G\cap \mathbb K[y_1,\dots,y_m]$. \hfill \qedsymbol
\end{thm} 

\begin{rem} The lexicographical orderings such that $x_1>\cdots >x_n$ and $x_n>\dots>x_1$ are  elimination orderings for every partition $\{x_{1},\ldots ,x_{k}\}\cup\{x_{k+1},\ldots ,x_{n}\}$  but the crucial difference for each partition of this type is that the first order is eliminating $\{x_{1},\ldots ,x_{k}\}$ and the second eliminates $\{x_{k+1},\ldots ,x_{n}\}$.
\end{rem}

\subsection{The divided power polynomial algebra and reduction modulo a prime \textit{p}}

Suppose $\mathsf{char}(\mathbb K) = 0$, let $\mathbb Z^\mathcal D_n  \subseteq \mathbb K_n$ be the divided power polynomial algebra in $n$ variables $x_0,\dots,x_{n-1}$ over $\mathbb Z$, i.e  the $\mathbb Z$-algebra generated by 
$$\left\{(x_i)^{(a_i)}:=\frac{x_i^{a_i}}{a_i!} \mid i=0,\dots,n-1, \ a_i\in \mathbb Z_+\right\}.$$  For any set of elements $S\subset \mathbb Z^\mathcal D_n$ we also denote by $\langle S \rangle$ the ideal generated by $S$ in $\mathbb Z^\mathcal D_n$. Any element of the form $x_0^{(a_0)}\dots x_{n-1}^{(a_{n-1})}$, with $(a_0,\dots,a_{n-1})\in \mathbb Z_+^n$, will also be called a monomial in $\mathbb Z^\mathcal D_n$. Further, any monomial order $\prec$ in $\mathbb K_n$ induces a monomial order in $\mathbb Z^\mathcal D_n$ in a obvious way and we keep denoting this induced order by $\prec$. Similarly, we also define $\mathbb Z^\mathcal D_\infty  \subseteq \mathbb K_\infty$.

The reduction of $\mathbb Z^\mathcal D_n$ modulo a prime $p$ is just the change of scalars of $\mathbb Z^\mathcal D_n$ by an algebraically closed field  $\mathbb F$ of characteristic $p$, that is  the tensor product $\mathbb Z^\mathcal D_n \otimes_\mathbb Z \mathbb F =:\mathcal D\mathbb F_n$ as $\mathbb Z$-modules. Equivalently, we can define $\mathcal D\mathbb F_n$ as the commutative algebra quotient $\mathbb F[x_i^{(k)}|0\leq i \leq n-1, k\geq 0]/I$, where
$$
I=\left \langle x_i^{(j)}x_i^{(k)}-\binom{j+k}{j}x_i^{(j+k)}|0\leq i \leq n-1, j,k\geq 0\right \rangle,
$$
which allows us to consider a similar Gr\"obner bases theory for $\mathcal D\mathbb F_n$. In fact it is known as Gr\"obner-Shirshov theory (for historical reasons) and we present it the next subsection.

\begin{rem} One of the most relevant theoretical difference between $\mathbb F_n$ and $ \mathcal D\mathbb F_n$ comes from the fact that $ \mathcal D\mathbb F_n$ is not Noetherian. 
\end{rem}

\subsection{Monomials associated to partitions} \label{partitions-monomials}

Following Subsection \ref{partitions}, we let $m_i(\lambda)$ denote the number of parts in $\lambda$ which are exactly equal to $i$. 
In certain cases, abusing notation, we may allow $\lambda$ to have a specified number of parts equal to $0$. In this case, is clear that for any partition $\lambda$ with $\ell(\lambda)=k\leq m-1$, we may associate the monomials $x^{(\lambda)}=x_0^{(m_0(\lambda))}x_1^{(m_1(\lambda))}\cdots x_k^{(m_k(\lambda))}\in \mathcal D\mathbb F_m$ and $x^{\lambda}=x_0^{m_0(\lambda)}x_1^{m_1(\lambda)}\cdots x_k^{m_k(\lambda)}\in \mathbb F_m$.  

\subsection{Gr\" obner-Shirshov bases}

Let $X=\{x_1,x_2,\dots\}$ be an enumerable set and let $X^*$ be the free monoid of associative monomials on $X$.
Fix a monomial order $\prec$ on $X^*$ and let $\mathbb F_X$ be the free associative algebra generated by $X$ over a field $\mathbb F$. Given a nonzero element $p\in \mathbb F_X$, we denote by $LM(p)$ the maximal monomial appearing in $p$ under the ordering $\prec$. In this case, $p=\alpha LM(p)+\sum_i \beta_iw_i$ with $\alpha,\beta_i\in \mathbb F$, $w_i\in X^*$, $\alpha\ne 0$ and $w_i\prec LM(p)$ for all $i$. If $\alpha = 1$, $p$ is said to be \textit{monic}.

Let $(S,T)$ be a pair of subsets of $\mathbb F_X$. Denote by $\langle S \rangle$ the ideal generated by $S$ in $\mathbb F_X$, and by $I_T$ the ideal of $\frac{\mathbb F_X}{\langle S\rangle}$ generated by the image of $T$ in $\frac{\mathbb F_X}{\langle S\rangle}$. We say that the algebra $A=\frac{\mathbb F_X}{\langle S\rangle}$ is \textit{defined by $S$} and the $A$-module $M=\frac{A}{I_T}$ is \textit{defined by the pair} $(S,T)$. 

A monomial $u\in X^*$ is said to be \textit{$(S,T)$-reduced} if $u\neq f LM(s)$ and $u\neq g LM(t)$ for any $s\in S$, $t\in T$ and $f,g\in X^*$. Otherwise, the monomial $u$ is said to be \textit{$(S,T)$-reducible}.

Let $p,q \in \mathbb F_X$, we define the composition of $p$ and $q$ as the polynomial $$S(p,q)=\frac{w_{p,q}}{LT(p)}p-\frac{w_{p,q}}{LT(q)}q$$
where $w_{p,q}=lcm(LM(p),LM(q))$. We say that $p,q \in \mathbb F_X$ are \textit{congruent with respect to the pair $(S,T)$}, and denote $p\equiv q \mod (S,T)$ if $p-q = \sum_i \alpha_i a_i s_i + \sum_j \beta_j b_j t_j$, where $\alpha_j, \beta_j\in \mathbb F$, $a_i, b_j\in X^*$, $s_j \in S$, $t_j\in T$, $a_iLT(s_i)\prec w_{p,q}$, and $b_jLT(t_j)\prec w_{p,q}$. When $T=\emptyset$, we simply write $p\equiv q \mod (S)$.

A pair $(S,T)$ of subsets of monic elements of $\mathcal A_X$ is called a \textit{Gr\"obner-Shirshov pair} if $S(p,q)\equiv 0 \mod (S)$ for any $p,q\in S$, $S(p,q)\equiv 0 \mod (S,T)$ for any $p,q\in T$, and $S(p,q)\equiv 0 \mod (S,T)$ for any $p\in S$ and $q\in T$.

\begin{thm}[\cite{KL,KLLO}] \label{gs} Let $(S,T)$ be a pair of subsets of monic elements in $\mathcal A_X$. Let $A=\mathcal A_X/\langle S\rangle$ be the associative algebra defined by $S$ and let $M=A/I_T$ be the $A$-module defined by $(S,T)$.
	
	\begin{enumerate}[(a)]
		\item The pair $(S,T)$ can be completed to a Gr\"obner-Shirshov pair $(\mathcal S,\mathcal T)$ for the $A$-module $M$.
	
		\item \label{gs2} If $(S,T)$ is a Gr\"obner-Shirshov pair for the $A$-module $M$, then the set of $(S,T)$-reduced monomials forms a linear basis of $M$.\hfill \qedsymbol
	\end{enumerate}
\end{thm}

\section{Graded Local Weyl modules}\label{mods}

We now recall the definition of graded local Weyl modules for hyper current algebras.  We refer to \cite{BMM,Csurvey,JMsurvey} for details on finite-dimensional representations, Weyl modules and related topics recently developed.

\

\textit{During this section, $\mathbb F$ will always denote an algebraically closed field of any characteristic.
}
\begin{defn}\label{weyl} Given $\lambda = \sum_{i=1}^r m_i \omega_i \in P^+$, the graded local Weyl module $W_\mathbb F(\lambda)$ is the $U_\mathbb F(\lie g[t])$-module generated by the element $v_\lambda$ with defining relations
	\begin{equation}\label{defrel}
	(x_\alpha^+\otimes t^r)^{(s)} v_\lambda= \Lambda_{\alpha_i,s}  v_\lambda  =  h- \lambda(h) v_\lambda=(x_\alpha^-)^{(k)}v_\lambda=0, 
	\end{equation}
	\text{for all} $h\in U_\mathbb F(\lie h),\ \alpha\in R^+, \  i\in I, \ r\ge 0,\ s>0,\ k>\lambda(h_\alpha)$.
\end{defn}

In other words, defining $\mathcal R$ as the left ideal of $U_\mathbb F(\lie g[t])$ generated by \begin{equation}\label{defrelideal}
(x_\alpha^+\otimes t^r)^{(s)}, \quad \Lambda_{\alpha_i,s},  \quad h- \lambda(h), \quad (x_\alpha^-)^{(k)}, 
\end{equation}
\text{for all} $h\in U_\mathbb F(\lie h),\ \alpha\in R^+, \  i\in I, \ r\ge 0,\ s>0,\ k>\lambda(h_\alpha)$, then 
\begin{equation}\label{defrelideal1}
W_\mathbb F(\lambda) = \frac{U_\mathbb F(\lie g[t])}{\mathcal R}.
\end{equation}

\begin{rem}
In the case $\mathbb F=\mathbb C$, we simply denote $W_\mathbb C(\lambda)=W(\lambda)$ following the traditional notation for the graded local Weyl module for the current algebra $\lie g\otimes \mathbb C[t]$.
\end{rem}

We recall the main result on local Weyl modules, that they are universal modules in the category of finite-dimensional $U_\mathbb F(\lie g[t])$-modules:

\begin{thm} \cite[Theorem 3.3.4]{BMM} For all $\lambda \in P^+$, the modules $W_\mathbb F(\lambda)$ are finite-dimensional and indecomposable. Moreover, any graded finite-dimensional $U_\mathbb F(\lie g[t])$-module generated by a vector $v_\lambda$ satisfying the relations
	$(x_\alpha^+\otimes t^r)^{(s)} v_\lambda  = (h- \lambda(h))v_\lambda=\Lambda_{\alpha_i,s} v_\lambda = 0, s\geq 1$, is a quotient of $W_\mathbb F(\lambda)$. \hfill \qedsymbol
\end{thm}

\section{A characteristic-free basis for Weyl modules for  \texorpdfstring{$U_\mathbb F (\lie{sl}_2[t])$}{} } \label{s:sl2}

From now we will consider only the case $\lie g=\lie{sl}_2$. Since $I$ is singleton for $\lie{sl}_2$, we shall denote $\omega=\omega_1$, $x^\pm = x_{\alpha_1}^\pm$, $h=h_1$, $\Lambda_{\alpha_1}=\Lambda$, and $\Lambda_{\alpha_1,k}=\Lambda_{k}$. Similarly, for $\alpha=\alpha_1$ we simply denote by $ X(u)$ the series $X_{\alpha}(u)$ and $Y(u_1,\dots, u_s)$ will denote $Y_\alpha(u_1,\dots,u_s)$, both defined in Section \ref{genGarland}. In this case $P^+$ is in a bijective correspondence with $\mathbb Z_+$ and we denote by $m$ the weight $m\omega$ and $v_m=v_{m\omega}$.

We now state some useful identities for local Weyl modules which will be relevant for our purposes. From  the defining relation of $W_\mathbb F(m)$ and Proposition \ref{prop:commutation} we get

\begin{gather} \label{fromBasicRel} \text{coeff}_{a_1,a_2,\dots,a_s}Y[s]^{(k-\ell(a))} v_m = 0 \text{ for } 0 \le \ell(a) \le k, \ k>m 
\end{gather}
where $\ell(a)=\sum_{i=1}^s a_i$.

Let $\cal J$ be the ideal of $U_\mathbb F(\lie n^-[t])$ generated by the set 
\begin{equation*}
 \left \{ \text{coeff}_{a_1,a_2,\dots,a_s}Y[s]^{(k-\ell(a))}|  0\leq \ell(a) \leq k,\ m+1 \leq k \right \}.
\end{equation*}

Since $W_\mathbb F(m) = U_\mathbb F(\lie n^-[t])v_m$, and $U_\mathbb F(\lie n^-[t])$ is commutative, it follows from \eqref{defrel} and \eqref{fromBasicRel} that $\cal J \cdot W_\mathbb F(m) = 0$. Further, by \eqref{PBWordering}, \eqref{defrelideal} and \eqref{defrelideal1}
\begin{equation}\label{l:quotient}
    W_\mathbb F(m) = \frac{U_\mathbb F(\lie g[t])}{\cal R} \cong \frac{U_\mathbb F(\lie n^-[t])}{ \cal J}. 
\end{equation}

In order to calculate the dimension and to construct a $\mathbb F$-basis of $W_\mathbb F(m)$ we will pass to a purely polynomial setting.

First, note that we have an isomorphism $ \phi: U_\mathbb F(\lie n^-[t]) \to \mathcal D\mathbb F_\infty$ defined by $(x^-\otimes t^r)^{(k)} \mapsto x_r^{(k)}$ for all $r,k\in \mathbb Z_+$. Now, consider the series $Y_\infty[s]$ with coefficients in $\mathcal D\mathbb F_\infty$ given by 
\begin{equation}
\sum_\eta (-1)^{\ell( \eta )}\frac{\ell(\eta)!}{\prod_{i=1}^{s}m_i(\eta)!}x_{|\eta|}\prod_{i=1}^{s} u_i^{m_i(\eta)} \label{formula:Y}
\end{equation}
where the sum is over all partitions $\eta$ of length $\leq s$. Respectively, we define $\mathsf{Y}[s]$ with coefficients in $\mathcal D\mathbb F_m$ by \eqref{formula:Y} with the restriction that $|\eta|\leq m-1$.

Consider the ideals
\begin{gather*}
J_\infty :=\left\langle  \text{coeff}_{a_1,a_2,\dots,a_s}Y_\infty[s]^{(k-\ell(a))}\ \mid \  0\leq \ell(a) \leq k,\ m+1 \leq k \right\rangle\\
J_m := \left\langle   \text{coeff}_{a_1,a_2,\dots,a_s}\mathsf{Y}[s]^{(k-\ell(a))}\ \mid \   0\leq \ell(a) \leq k,\ m+1 \leq k \right\rangle.
\end{gather*}
From the isomorphism $\phi$ and the inclusion $\mathcal D\mathbb F_m \hookrightarrow \mathcal D\mathbb F_\infty$, we get isomorphisms 
\begin{equation}\label{e:isom1}
\frac{\mathcal D\mathbb F_m}{ J_m} \cong \frac{\mathcal D\mathbb F_\infty}{ J_\infty} \cong  \frac{U_\mathbb F(\lie n^-[t])}{\cal J}.
\end{equation}
In particular, \eqref{l:quotient} and \eqref{e:isom1} implies $$\dim W_\mathbb F (m) = \dim \frac{\mathcal D\mathbb F_m}{J_m}$$
as $\mathbb F$-vector spaces.

The rest of this section is devoted to showing that in fact we have $\dim W_\mathbb F (m)=2^m$ (in particular, the dimension is independent of the ground field), and to construct an explicit basis for $W_\mathbb F (m)$. 

We now discuss a Gr\"obner-Shirshov pair of $J_m$. First, we need to define lexicographic ordering for the divided power algebra $\mathcal D\mathbb F_m$. This is not the same as the usual lexicographic order, but it is adapted for reducing our infinite alphabet $\{x_i^{(j)}|0\leq i\leq m-1, j\geq 1\}$ modulo the relations $x_i^{(r)}x_i^{(s)}-\binom{r+s}{s}x_i^{(r+s)}$. 
First, we write down a monomial in increasing order of its variables and for each variable in increasing order of its upper index. For monomials of the form $x_i^{(\mu_1)}x_i^{(\mu_2)}\cdots x_i^{(\mu_k)}$ with $\mu$ a partition, we compare them by degree-graded lexicographic ordering, then we extend this ordering to a product ordering on arbitrary monomials by using lexicographic order on the index of variables appearing in the monomials. This is a well-defined monomial ordering, and it can be seen as a natural extension of the usual lexicographic ordering.

\begin{exe} Let $\lambda$ and $\mu$ two partitions as in Subsection \ref{partitions} and \ref{partitions-monomials}.
	
	\begin{enumerate}
		\item If $|\lambda|>|\mu|$, then $x_i^{(\lambda)}$ is greater than $x_i^{(\mu)}$ for any $0\le i\le m-1$. Further, $x_i^{(k)}$, $0\le i\le m-1$ and $k>0$, is greater than any polynomial involving only $x_j$ with $j>i$.
		
		\item $(x_i^{(1)})^{m_1(\lambda)}(x_i^{(2)})^{m_2(\lambda)}\cdots (x_i^{(\lambda_1)})^{m_{\lambda_1}(\lambda)}\geq_{\text{lex}}x_i^{(|\lambda|)}$. In other words, $x_i^{(|\lambda)|}$ is the minimum monomial over all partitions $\mu\vdash |\lambda|$.
		
		\item $x^{(\lambda)}<_{\text{lex}}x^{(\mu)}$ if, and only if, $\lambda <_{\text{lex}} \mu$ using the usual lexicographic order for partitions, i.e. $m_i(\lambda)< m_i(\mu)$ where $i$ is the least index such that $m_i(\lambda)\neq m_i(\mu)$.
		
	\end{enumerate} 
\end{exe}

Before we continue, notice that to find an $\mathbb F$-basis for $\frac{\mathcal D\mathbb F_m}{J_m}$ by using Gr\"obner-Shirshov theory,  it is not necessary to compute a Gr\"obner-Shirshov basis for $\frac{\mathcal D\mathbb F_m}{J_m}$ completely. 
Recall that, if $I$ is an ideal of $\mathcal D\mathbb F_m$ with Gr\"obner-Shirshov pair $(S,G)$ with respect to a monomial order, then a linear basis of $\frac{\mathcal D\mathbb F_n}{I}$ is the set of unreduced monomials with respect to $(S,G)$. 

The main theorem of this paper is to find a Gr\"obner-Shirshov pair $(S,G_m)$ for $J_m$ with respect to the ordering defined above, and thereby to find a basis of $\frac{\mathcal D \mathbb F_m}{J_m}$, where $$S=\left \langle x_i^{(j)}x_i^{(k)}-\binom{j+k}{j}x_i^{(j+k)}|0\leq i \leq n-1, j,k\geq 0\right \rangle$$ and the set $G_m$ will be given later, but for now we give only the set of leading monomials (which also determine the basis elements). We are ready to state our main result:

\begin{thm}\label{t:leader}
	The set of leading monomials of $G_m$ is 
	\begin{equation*}
	\{ x_0^{(a_0)}\cdots x_{s}^{(a_{s})} \mid 0\leq s \leq m-1, a_0+a_1+\cdots+a_{s}>m-s, a_{s}\ne0 \}.
	\end{equation*}
\end{thm}

\begin{cor} \label{basis} The set $\{ x_0^{(a_0)}\cdots x_{s}^{(a_{s})} \mid 0\leq s \leq m-1, a_0+a_1+\cdots+a_{s}\leq m-s\}$  is a monomial $\mathbb F$-linear basis of $\frac{\mathcal D \mathbb F_m}{J_m }$. In particular, $\dim\left(\frac{\mathcal D\mathbb F_m}{ J_m }\right) = \dim W_\mathbb F (m) = 2^m$. 
\end{cor}

\proof First part follows from Theorems \ref{t:leader} and \ref{gs}\eqref{gs2}. Last part is immediate by counting the elements given in the first part.
\endproof

Putting in the original notation we get:

\begin{cor} \label{cor:basis} The set $\{ (x^-\otimes 1)^{(a_0)}\cdots (x^-\otimes t^s)^{(a_{s})} v_m \mid  0\leq s \leq m-1, a_0+a_1+\cdots+a_{s}\leq m-s\}$  is a monomial $\mathbb F$-linear basis of $W_\mathbb F (m)$. \hfill \qedsymbol
\end{cor}

\begin{rem} The basis described here is the same basis obtained by Chari-Pressley \cite{CPweyl} and by Kus-Littelmann \cite{kus} in the case of $W_\mathbb C(m)$. The present construction can be thought as a third construction of the same basis for $W_\mathbb C(m)$.  However, it never appeared in the literature in the hyper current context (positive characteristic case), which makes relevant the achievement. 
\end{rem}

\section{Some applications}\label{app}

\textit{In the next two subsections we suppose $\mathbb F$ is an algebraically closed field of characteristic zero. }

\subsection{The basis for \texorpdfstring{$W_\mathbb F (m)$}{} coming from the reverse lexicographic order.}
The simple combinatorial description of the ideal $J_m$ and its leading terms with respect to a fixed order is also useful to construct bases from other monomial orders in $\mathbb F_m$. We can  describe a basis for $ W_\mathbb F (m)$  whenever we can describe the leading term of the elements in a Gr\"{o}bner basis for $J_m$ with respect to this order similarly to the method of the famous \textit{Gr\"{o}bner Walk algorithm}.

The calculation and procedure to construct the basis with respect to the reverse lexicographic order through the basis in the lexicographic case is very technical and we present it separately in Subsection \ref{s:revbasis}. We give here only the statement of our second main result. The explanation for disregarding positive characteristic will be given then.

\begin{thm} \label{t:revbasis} 
	Let 
	\begin{equation*}
	\mathcal R_1 = \left \{\begin{tabular}{l|l}
	  $x_0^{f_0}x_1^{f_1}\cdots x_{m-1}^{f_{m-1}} v_m $&\begin{tabular}{l} $  f_i\geq 0,\sum_{i=0}^{m-1}f_i\leq \frac{m}{2},$\\
	  $(i-1)f_i+i f_{i+1} \leq m-2\sum_{j=i}^{m-1} f_j,$\\ for $ 1\leq i \leq m-2 $ \end{tabular}\end{tabular}\right \},
	 \end{equation*} 	 
	 and $$\mathcal R_2= \left\{x_0^{f_0}x_1^{f_1}\cdots x_{m-1}^{f_{m-1}}v_m \mid   f_i\ge 0, \sum_{i=0}^{m-1}f_i=k> \frac{m}{2}  \text{ and }  x_0^{m-2k+f_0}x_1^{f_1}\cdots x_{m-1}^{f_{m-1}} \in \mathcal R_1 \right\}.$$
	 The set $\mathcal R_m = \mathcal R_1 \cup\mathcal R_2$ is a basis for $W_\mathbb F(m)$. 
\end{thm}

\subsection{} The Chari-Venkatesh construction in \cite{CV} provides a basis for $W_\mathbb C(m)$ as follows (we are suiting their original notation to correspond to ours). Let $\mathcal S_m$ be the set of $m$-tuples $(i_0,\dots,i_{m-1})\in (\mathbb  Z^+)^m$ such that for all $0\le k\le m-1$ and $1\le j\le k+1$ we have $$ji_{k}+(j+1)i_{k+1}+2\sum_{p=k+2}^{m-1}i_p \le m-k+j+1.$$
Then the elements 
\begin{equation}\label{cv}
\{(x^-\otimes 1)^{i_0}\dots (x^-\otimes t^{m-1})^{i_{m-1}}\mid (i_0,\dots,i_{m-1})\in \mathcal S_m\}
\end{equation}
form a basis for $W_\mathbb F(m)$ (cf. \cite[Theorem 5]{CV}).

Surprisingly this basis and the basis coming from the reverse lexicographic order with our Gr\"obner basis approach are the same:

\begin{prop} The Chari-Vankatesh basis  and the basis of Theorem \ref{t:revbasis} are the same.
\end{prop}

\proof Let $\mathbf{i}=(i_0,i_2,\dots,i_{m-1})$ and  $|\mathbf{i}|=\sum_{p=0}^{m-1} i_p\le \frac{m}{2}$. Then, for $0\leq k\leq m-1$, and $1\le j\le k+1$:
\begin{equation}
ji_{k}+(j+1)i_{k+1}+2\sum_{p=k+2}^{m-1} i_p\le m-k+j+1 \label{eq:CVcondition}
\end{equation}
In particular, for $j=k+1$ we have:
$$(k+1)i_{k}+(k+2)i_{k+1}+2\sum_{p=k+2}^{m-1} i_p\le m$$
which implies 
$$(k-1)i_{k}+ki_{k+1}+2\sum_{p=k}^{m-1} i_p\le m.$$
Therefore, $x^{\mathbf{i}}\in \mathcal{R}_1$. Now, let $\mathbf{i}=(i_0,i_2,\dots,i_{m-1})\in \mathcal S_m$ satisfy  $r=|\mathbf{i}|>\frac{m}{2}$. Then, by (\ref{eq:CVcondition}) with $k=0$, $j=1$ we have:
\begin{equation}
i_0+2i_1+2\sum_{p=2}^{m-1} i_p\le m. \label{eq:symm_cond}
\end{equation}

Therefore:
\begin{align*}
m-2r+i_0+2\sum_{p=1}^{m-1}i_p&\le 2(m-r) < 2\left ( \frac{m}{2}\right )=m
\end{align*}
and
\begin{align*}
m-2r+i_0&=m+i_0-2\sum_{p=0}^{m-1}i_p=m-i_0-2\sum_{p=1}^{m-1}i_p\ge 0,
\end{align*}
where the last inequality follows from condition \ref{eq:symm_cond}. Hence, $(m-2r+i_0,i_1,\dots,i_{m-1})\in \mathcal S_m$, since $i_0$ only enters (\ref{eq:CVcondition}) when $k=0$.

Finally,
\begin{align*}
m-2r+i_0+\sum_{p=1}^{m-1}i_p=m-2r+r=m-r\leq \frac{m}{2}
\end{align*}
Therefore,  we have $x^{m-2r+i_0}x_1^{i_1}\cdots x_{m-1}^{i_{m-1}}\in \mathcal R_1$. Hence, $x^\mathbf{i}\in \mathcal R_2$.

We conclude that the Chari-Venkatesh basis is a subset of our basis. In particular, since they are both bases, we see that both sets are equal.
\endproof

\begin{rem}
In view of this proposition, it is simpler to consider the presentation obtained by Chari-Venkatesh (cf. \eqref{cv}) than the presentation in Theorem \ref{t:revbasis}, remembering its linkage with the reverse lexicographical order. The advantage of our interpretation becomes clear from its use in Theorem \ref{truncbasis}. The basis has some interesting properties:

\begin{enumerate}
	\item  It differs from the basis for $W_\mathbb C(m)$ presented in Corollay \ref{cor:basis} (also Chari-Pressley and Kus-Littelmann bases).
	
	\item  By letting $\Theta_m$ be the index set of elements of $\mathcal R_m$, the combinatorial description allow us to conclude that these sets are \textit{well-behaved} with respect to inclusions:
	$$\Theta_0\subseteq \Theta_1\subseteq \dots \subseteq \Theta_m.$$
	
	\item It shows the combinatorial skeleton of the $\lie{sl}_2$--module structure. Monomials not divisible by $x_0$ correspond bijectively with $\lie{sl}_2$--highest weight vectors. Such a monomial not involving $x_0$ has weight $m-k+1$ where $k$ is its degree. The sets $\mathcal R_1$ and $\mathcal R_2$ realize the symmetry between positive and negative $\lie{sl}_2$--weight spaces.
	
	\item The next application is a good reason to consider this basis as special and also \textit{well-behaved} with respect to truncations.
\end{enumerate}
	
\end{rem}

\subsection{A basis for truncated local Weyl modules for \texorpdfstring{$\lie{sl}_2$}{} } Let $N\in\mathbb N$. For any Lie algebra $\lie g$ as in Subsection \ref{pre}, set $\lie{g}[t]_N  = \lie g  \otimes \mathbb F [t]/t^N\mathbb F[t] $. We recall the definition of the $N$-truncated local Weyl module for $U_\mathbb F(\lie g[t]_N)$.

\begin{defn}\label{weyltrunc} Given $\lambda = \sum_{i=1}^r m_i \omega_i \in P^+$, the $N$-truncated local Weyl module $W_\mathbb F(\lambda,N)$ is the $U_\mathbb F (\lie g[t]_N)$-module generated by the element $v_{\lambda,N}$ with defining relations
	\begin{gather*}
	(\lie n^+\otimes \mathbb F[t]/t^{N}\mathbb F[t] ) \ v_{\lambda,N} = 	(\lie h\otimes t\mathbb F[t]/t^{N}\mathbb F[t]) \ v_{\lambda,N} = 0 \\ (h- \lambda(h)) \ v_{\lambda,N}=(x_\alpha^-)^{(k)} \ v_{\lambda,N}=0,
	\end{gather*}
	\text{for all} $h\in U(\lie h),\ \alpha\in R^+,\ k>\lambda(h_\alpha)$. 
\end{defn}

Notice that $\lie g[t]_N \cong \frac{\lie g[t]}{\lie g\otimes t^n\mathbb F[t]}$ and $W_\mathbb F(\lambda,N)$ naturally becomes a $\lie g[t]$-module. The universal properties of Weyl modules
gives us epimorphisms of $\lie{g}[t]$-modules
$$W_\mathbb F(\lambda) \twoheadrightarrow W_\mathbb F(\lambda, N) \qquad \text{ and } \qquad W_\mathbb F(\lambda,N) \twoheadrightarrow W_\mathbb F(\lambda, N') \text{ for } N\ge N'.$$
It is now easy to see that relevant truncations are given by $N<\lambda (h_\alpha)$, for $\alpha \in R^+$, since 
$$W_\mathbb F(\lambda) \cong W_\mathbb F(\lambda, N) \text{ for } N\ge \max\{ \lambda(h_\alpha)\mid \alpha \in R^+\}.$$

\

In particular, if $\lie g=\lie{sl}_2$ and $\lambda=m\omega_1$, denoting $v_{\lambda,N}$ by $v_{m,N}$, we get $$W_\mathbb F(m,N)\cong \frac{W_\mathbb F(m)}{\langle x^-\otimes t^N \ v_{m,N}\rangle}
.$$

We can use Theorem \ref{t:revbasis}  to get a basis for $W_\mathbb F(m,N)$: 

\begin{thm}\label{truncbasis} If $N<m$, the set of elements in $\mathcal R_m$ which do not involves $x_N,\dots,x_{m-1}$ forms a basis for  $W_\mathbb F(m,N)$ for $U_\mathbb F (\lie{sl}_2[t]_N)$.   
\end{thm}

\proof
It directly follows from \eqref{e:isom1} and the Elimination Theorem \ref{elim} by setting $Y=\{x_0,\dots,x_{N-1}\}$ and $X=\{x_N,\dots,x_{m-1}\}$ and considering the reverse lexicographic order $x_0<\dots<x_{N-1}<x_N<\dots<x_{m-1}$.
\endproof

\begin{rem}
The Kus-Littelmann basis for the truncated local Weyl module $W_\mathbb C(m,N)$ described in \cite{kus} is different from the basis constructed here. The construction in that paper uses an advanced realization of these Weyl modules as \textit{fusion products} and it is an intrinsic construction. The crucial difference comes from the fact that our procedure is based on the construction of a specific basis for graded local Weyl modules (with no truncation) and then we cut off some elements in a very simple way. 
In particular, by letting $\Theta_{m}^N$ be the index set of elements of a basis for $W_\mathbb F(m,N)$, we conclude that we have the following natural chain of inclusions:
$$\Theta_m^1\subseteq \Theta_m^2\subseteq \dots \subseteq \Theta_m^{m-1}\subseteq \Theta_m.$$

\end{rem}

\section{Proofs of the main results}\label{proofs}

In this section, we will prove the main results.

\subsection{Proof of Theorem \ref{t:leader}}
\textit{In what follows all unadorned tensor product symbol $\otimes$ are considered over the field $\mathbb F$ (i.e. $\otimes_\mathbb F$).}
One of the main problems in module theory is showing when a given ring element acts as zero on a given module element. Sending the module in question to another, more well-understood module can give insight into this problem. Therefore, we use the following helpful construction from \cite{CPweyl}, now adapted to the positive characteristic setting, as a guide.
 
Define $V=\text{span}_{\mathbb F}\{v_+,v_-\}$. We give $V\otimes \mathbb F[t]$ a $U_{\mathbb F}(\lie{sl}_2[t])$-module structure by the following:
\begin{gather*}
(x^- \otimes t^k)\cdot( v_+ \otimes t^r) = v_- \otimes t^{k+r}, \qquad (x^+ \otimes t^k)\cdot( v_- \otimes t^r) = v_+ \otimes t^{k+r},\\
(x^- \otimes t^k)\cdot( v_- \otimes t^r) = 0, \qquad (x^+ \otimes t^k)\cdot( v_+ \otimes t^r) = 0,\\
(x^- \otimes t^k)^{(s)}\cdot(V\otimes \mathbb F[t]) =  (x^+ \otimes t^k)^{(s)}\cdot (V\otimes \mathbb F[t]) = 0,\\ 
\text{ for all } s>1, r\ge 0, \text{ and } k\ge0.
\end{gather*}
Thus we also consider the $U_\mathbb F (\lie{sl}_2[t])$-module $(V\otimes \mathbb F[t])^{\otimes m}$ with the usual tensor product action:
\begin{gather*}
(x^\pm\otimes t^k)^{(r)}\cdot (v_1\otimes v_2)=\sum_{i=0}^{r}(x^\pm\otimes t^k)^{(i)}\cdot v_1\otimes (x^\pm\otimes t^k)^{(r-i)}\cdot v_2,\\
\Lambda_r\cdot (v_1\otimes v_2)=\sum_{i=0}^{r}\Lambda_i\cdot v_1\otimes \Lambda_{r-i}\cdot v_2,\\
\binom{h}{r}\cdot (v_1\otimes v_2)=\sum_{i=0}^{r}\binom{h}{i}\cdot v_1\otimes \binom{h}{r-i}\cdot v_2.
\end{gather*}
If $M$ is a $U_\mathbb F (\lie{sl}_2[t])$-module then by $(M^{\otimes k})^{\Sigma_k}$ we mean the $U_\mathbb F (\lie{sl}_2[t])$-submodule of $M^{\otimes k}$ of invariants under permutation of tensor factors. Also, note that 
$$ (V \otimes  \mathbb F[t])^{\otimes m} \twoheadrightarrow ((V\otimes  \mathbb F[t])^{\otimes m})^{\Sigma_m}$$
under the homomorphism $\text{Sym}$, which maps a tensor to the sum of all \emph{distinct} permutations of the positions of the vectors. The $U_\mathbb F (\lie{sl}_2[t])$-module action on $((V \otimes \mathbb F[t])^{\otimes m})^{\Sigma_m}$ induces an isomorphism
$$ ((V \otimes \mathbb F[t])^{\otimes m})^{\Sigma_m}  \cong (V^{\otimes m} \otimes \mathbb F[t_1,\dots,t_m])^{\Sigma_m}$$
which maps $\text{Sym}((v_1\otimes t^{r_1})\otimes(v_2\otimes t^{r_2})\otimes \cdots\otimes (v_k\otimes t^{r_k}))$ to $\text{Sym}(v_1\otimes v_2\otimes \cdots \otimes v_k) \otimes M_{(r_1,r_2,\dots, r_k)}(t_1,t_2,\cdots, t_m)$. Here and further $M_{\lambda}(t_1,t_2,\dots, t_m)$ is the monomial symmetric polynomial in $m$ variables associated with the partition $\lambda$. Now, consider the submodule 
$$W_m:=\left (V^{\otimes m}\otimes \mathbb F[t_1,\dots,t_m]_{\Sigma_m}\right )^{\Sigma_m}$$
where 
$$\mathbb F[t_1,\dots,t_m]_{\Sigma_m}=\frac{\mathbb F[t_1,\dots,t_m]}{\langle\mathbb  F[t_1,\dots,t_m]^{\Sigma_m}_+\rangle}$$
is the coinvariant algebra of $\Sigma_m$.

There is a $\mathcal D \mathbb F_m$-module homomorphism $\mathcal D \mathbb F_m  \stackrel{\psi}{\rightarrow} W_m$ defined by $1 \mapsto v_+^{\otimes m}\otimes 1$, where $W_m$ is considered a $\mathcal D\mathbb F_m$-module by $\mathcal D \mathbb F_m \hookrightarrow \mathcal D\mathbb F_\infty \cong U_{\mathbb F}(\lie n^-[t])$ as discussed in Section \ref{s:sl2}. We here and further assume the polynomial in the second tensor factor to be reduced in $\mathbb F[t_1,\dots,t_m]_{\Sigma_m}$. We can prove that $\psi$ is surjective using an argument similar to \cite[Lemma 6.3]{CPweyl}.

Let $\lambda$ be a partition satisfying $\lambda_1\leq m-1$ and $\ell(\lambda)=r\leq m$, where 0 may be included as a part. Let $x^{(\lambda)}=x_0^{(j_0)}x_1^{(j_1)}\cdots x_{m-1}^{(j_{m-1})}$ with $j_k=m_k(\lambda)$ be the monomial corresponding to $\lambda$ in $\mathcal D\mathbb F_m$. Let $\Sigma^r_{j_0,j_1,\dots,j_{m-1}}$ be the subgroup of all $\rho\in\Sigma_r$ satisfying $\rho(s_k+1)<\rho(s_k+2)<\cdots<\rho(s_{k+1}),k=0,1,\dots, m-2$ where $s_k=\sum_{i=0}^{k-1} j_i$ and $s_0=0$.  This is called the group of $(j_0,j_1,\dots,j_{m-1})$-shuffles of $r$. We compute 
\begin{align*}
\psi(x^{(\lambda)})&=x_0^{(j_0)}x_1^{(j_1)}\cdots x_{m-1}^{(j_{m-1})}\cdot (v_+\otimes 1)^{\otimes m}\\
&= \sum_{\rho \in \Sigma^m_{j_{m-1},\dots,j_{0},m-r}}\rho((v_-\otimes t^{m-1})^{\otimes j_{m-1}} \otimes \cdots \otimes (v_-\otimes 1)^{\otimes j_{0}}\otimes (v_+\otimes 1)^{\otimes m-r})\\
&=\sum_{\rho \in \Sigma^m_{j_{m-1},\dots,j_{0},m-r}}\rho(v_-^{\otimes r}\otimes v_+^{\otimes m-r}\otimes t_1^{\lambda_1}t_2^{\lambda_2}\cdots t_r^{\lambda_r})\\
&= \sum_{\sigma \in \Sigma^m_{r,m-r}}\left (\sigma(v_-^{\otimes r}\otimes v_+^{\otimes m-r})\otimes \sum_{\pi \in \Sigma^r_{j_{m-1},j_{m-2},\dots,j_{0}}} t_{\sigma(\pi(1))}^{\lambda_1}t_{\sigma(\pi(2))}^{\lambda_2}\cdots t_{\sigma(\pi(r))}^{\lambda_{r}}\right )\\
&= \sum_{\sigma \in \Sigma^m_{r,m-r}}\sigma(v_-^{\otimes r}\otimes v_+^{\otimes m-r})\otimes M_{\lambda}(t_{\sigma(1)},t_{\sigma(2)},\dots, t_{\sigma(r)}).
\end{align*}

The $r$ variable Schur function associated to $\lambda$ is the symmetric polynomial defined by
\begin{equation}
s_\lambda(t_1,t_2,\dots,t_r)=\sum_{\mu \preceq \lambda}K_{\lambda\mu}M_{\mu}(t_1,t_2,\dots, t_r),
\end{equation}
where the positive integers $K_{\lambda\mu}$ are the Kostka numbers (see \cite{Mac}), satisfying $K_{\lambda\lambda}=1$. We also define, letting $\lambda$ be a partition with at most $r$ parts with highest part $\lambda_1<m$, which we extend to have $r$ parts by adding 0 where necessary:
$$
s_{\lambda,r}(x_0,x_1,\dots,x_{m-1})=\sum_{\mu \preceq \lambda}K_{\lambda\mu}x^{(\mu)} 
$$
and compute
\begin{align*}
\psi (s_{\lambda,r}(x_0,x_1,\dots,x_{m-1}))&=\text{Sym}\left (v_-^{\otimes r}\otimes v_+^{\otimes m-r}\otimes \sum_{\mu \preceq \lambda}K_{\lambda\mu}M_{\mu}(t_1,t_2,\dots, t_r)\right )\\
&=\text{Sym}\left (v_-^{\otimes r}\otimes v_+^{\otimes m-r}\otimes s_\lambda(t_1,t_2,\dots, t_r)\right ).
\end{align*}

The complete homogeneous symmetric polynomial in $r$ variables is defined by
\begin{equation*}
h_k(t_1,t_2,\dots, t_r)=\sum_{j_1+j_2+\cdots+j_r=k; \ j_i\geq 0}t_1^{j_1}t_2^{j_2}\dots t_{r}^{j_{r}}.
\end{equation*}
They are important to us because of the following result, which gives us our first application of Gr\"obner basis theory:

\begin{thm}[\cite{MS}]
For any field, the set of polynomials $\{h_{m-r+1}(t_1,t_2,\dots, t_{r})|1\leq r\leq m\}$
is a Gr\"obner basis of $\mathbb F[t_1,\dots,t_m]_{\Sigma_m}$ with respect to lexicographic order such that $t_1<t_2<\dots<t_m$. \hfill \qedsymbol
\end{thm} 

\

\begin{prop}
	Fix $0\leq k\leq m$. Let $\lambda$ be a partition with highest part $\lambda_1\leq m-k$ and $\ell(\lambda)\leq k$. Then $$\psi(s_{\lambda,k}(x_0,x_1,\dots, x_{m-1}))\neq 0.$$
\end{prop}
\begin{proof}
We have $\psi (s_{\lambda,r}(x_0,x_1,\dots,x_{m-1}))=0$ if, and only if, 
$s_\lambda(t_1,t_2,\dots, t_k)= 0 \in \mathbb F[t_1,t_2,\dots, t_m]_{\Sigma_m}$. The leading term of $s_\lambda(t_1,t_2,\dots, t_k)$ is $t_{1}^{\lambda_{k}}t_{2}^{\lambda_{k-1}}\cdots t_k^{\lambda_1}.$ For $r\leq k$, consider $h_{m-r+1}(t_1,t_2,\dots, t_r)$ whose leading monomial is $t_r^{m-r+1}$. Since $m-r+1\geq m-k+1>\lambda_1$, $t_r^{m-r+1}$ does not divide the leading monomial of $s_\lambda$. But if $r>k$ then $t_r$ does not appear in $s_\lambda(t_1,t_2,\dots, t_k)$. Therefore, $\psi (s_{\lambda,k}(x_0,x_1,\dots,x_{m-1}))\neq 0$ with the given conditions on $\lambda$.
\end{proof}
Let $k \in \{2,\dots, m+1\}$ and $\lambda$ be a partition such that all the parts of $\lambda$ are $\leq m-1$. Here, $\ell(\lambda)\leq m+1$ and no other restriction is made. Note that $\lambda$ is independent of $k$. Define $f_{\lambda,k}$ to be the following homogeneous function of degree $k$ (the sum is over partitions $\mu$ of length $k$ each of whose parts is $\leq m-1$ with 0 possibly included):
$$f_{\lambda, k}(x_0,x_1,\dots,x_{m-1})=\sum_{\mu\succeq\lambda} D_{\lambda \mu} x^{(\mu)},$$
where $$ D_{\lambda\mu}=(-1)^{|\mu|-\ell(\lambda)}\sum_{\eta^{1},\eta^{2},\dots, \eta^{\ell(\mu)}}\prod_{i=1}^{\ell(\mu)} 
\frac{\ell(\eta^{i})!}{\prod_{j=1}^{\eta_1^{i}}m_j(\eta^{i})!}$$
and the sum is over all sequences of partitions such that $\eta^{1}\uplus\eta^{2}\uplus \cdots \uplus \eta^{\ell(\mu)}=\lambda$ and $\eta^{i}\vdash \mu_i$. The notation $\eta^1 \uplus \eta^2$ denotes the partition $(\eta^1_1,\eta^1_2,\dots,\eta^1_{\ell(\eta^1)},\eta^2_1,\eta^2_2,\dots, \eta^2_{\ell(\eta^2)})$ with terms arranged in decreasing order. Observe that $\mu \succeq \lambda$ implies $\mu\leq_{\text{revlex}}\lambda$. The $D_{\lambda\mu}$ appear in the following expression for the ``forgotten symmetric polynomials'' (see \cite[page 22]{Mac}, \cite[Equation (2.41)]{Z}):

$$f_\lambda(t_1,t_2,\dots)=\sum_{\mu \succeq \lambda}D_{\lambda\mu}M_\mu(t_1,t_2,\dots).$$

We now use Proposition \ref{prop:commutation} to prove:

\begin{prop}\label{prop:basis_funs}
$f_{\lambda,k}(x_0,x_1,\dots,x_{m-1})$ is in $J_m$ if $\ell(\lambda)>m-k$.
\end{prop}

\begin{proof}
We want to show that $f_{\lambda,k}(x_0,x_1,\dots,x_{m-1})$ is, up to a sign, given by the coefficient of $u_1^{m_1(\lambda)}u_2^{m_2(\lambda)}\cdots u_{\lambda_1}^{m_{\lambda_1}(\lambda)}$ in $\mathsf{Y}[\lambda_1]^{(k)}$. By expanding, we have
\begin{align*}
\mathsf{Y}[\lambda_1]^{(k)}&=\left(\sum_\eta (-1)^{\ell( \eta )}\frac{\ell(\eta)!}{\prod_{i=1}^{\lambda_1}m_i(\eta)!}x_{|\eta|}\prod_{i=1}^{\lambda_1} u_i^{m_i(\eta)}\right )^{(k)}\\
&=\sum_{\{\nu\vdash k\}}\sum_{\eta^1,\eta^2,\dots,\eta^{\ell(\nu)}}\prod_{j=1}^{\ell(\nu)}\left ((-1)^{\ell( \eta^j )}\frac{\ell(\eta^j)!}{\prod_{i=1}^{\lambda_1}m_i(\eta^j)!}x_{|\eta^j|}\prod_{i=1}^{\lambda_1} u_i^{m_i(\eta^j)}\right)^{(\nu_j)}\\
&=\sum_{\{\nu\vdash k\}}\sum_{\eta^1,\eta^2,\dots,\eta^{\ell(\nu)}}\prod_{j=1}^{\ell(\nu)}\left ((-1)^{\ell( \eta^j )}\frac{\ell(\eta^j)!}{\prod_{i=1}^{\lambda_1}m_i(\eta^j)!}\right)^{\nu_j}\prod_{j=1}^{\ell(\nu)}x_{|\eta^j|}^{(\nu_i)}\prod_{j=1}^{\ell(\nu)}\prod_{i=1}^{\lambda_1} u_i^{m_i(\eta^j)\nu_j}.
\end{align*}
In the above, the inner sum is over distinct partitions. Now we change variables by letting $\mu$ be such that $m_{|\eta_j|}(\mu)=\nu_j$ and $({\eta'}^j)_{j=1}^{\ell(\mu)}$ be the sequence of partitions resulting taking $\eta^j$ repeated $m_{|\eta_j|}(\mu)$ times. We get
\begin{equation*}
\sum_{\{\mu|\ell(\mu)=k\}}\sum_{{\eta'}^1,{\eta'}^2,\dots,{\eta'}^{\ell(\mu)}}\prod_{j=1}^{\ell(\mu)}\left ((-1)^{\ell( {\eta'}^j )}\frac{\ell({\eta'}^j)!}{\prod_{i=1}^{\lambda_1}m_i({\eta'}^j)!}\right)x^{(\mu)}\prod_{i=1}^{\lambda_1} u_i^{\sum_{j=1}^{\ell(\mu)}m_i({\eta'}^j)}.
\end{equation*}
It is important to note that $\mu$ may contain $0$ in the sum above. We take the coefficient of the term $u_1^{m_1(\lambda)}u_2^{m_2(\lambda)}\cdots u_{\lambda_1}^{m_{\lambda_1}(\lambda)}$. This implies that, for $1\leq i \leq \lambda_1$:
\begin{align*}
m_i(\lambda)&=\sum_{j=1}^{\ell(\mu)}m_i({\eta'}^j).
\end{align*}
Therefore, $\lambda={\eta'}^1\uplus{\eta'}^2\uplus\cdots\uplus{\eta'}^{\ell(\mu)}$, and $|\mu|=|\lambda|$. Hence, the coefficient of $u_1^{m_1(\lambda)}u_2^{m_2(\lambda)}\cdots u_{\lambda_1}^{m_{\lambda_1}(\lambda)}$ is
\begin{equation*}
\sum_{\{\mu|\ell(\mu)=k\}}(-1)^{\ell( \lambda )}\sum_{{\eta'}^1,{\eta'}^2,\dots,{\eta'}^{\ell(\mu)}}\prod_{j=1}^{\ell(\mu)}\left (\frac{\ell({\eta'}^j)!}{\prod_{i=1}^{\lambda_1}m_i({\eta'}^j)!}\right)x^{(\mu)},
\end{equation*}
where the inner sum is over all ${\eta'}^1,{\eta'}^2,\dots,{\eta'}^{\ell(\mu)}$ such that ${\eta}'^j\vdash\mu_j$ and $\lambda={\eta'}^1\uplus{\eta'}^2\uplus\cdots\uplus{\eta'}^{\ell(\mu)}$. But this is just:
\begin{equation*}
\sum_{\{\mu|\ell(\mu)=k\}}(-1)^{|\lambda|}D_{\lambda\mu}x^{(\mu)}=(-1)^{|\lambda|}f_{\lambda,k}(x_0,x_1,\dots,x_{m-1}).
\end{equation*}
By Proposition \ref{prop:commutation}, elements in $J_m$ correspond to coefficients of $u^{\lambda}$ in $\mathsf Y[\lambda_1]^{(k)}$ where $k+\ell(\lambda)>m$.
\end{proof}

The transition matrix from $f_\mu$ to $s_\lambda$ (\cite[Table 1]{Mac}) gives:
\begin{equation} 
s_\lambda(t_1,t_2,\dots,t_{k})=\sum_{\mu\preceq\lambda'} K_{\lambda'\mu} f_{\mu}(t_1,t_2,\dots,t_k).\label{eq:transition}
\end{equation}
The identity in \cite[page 22]{Mac} holds in infinitely many variables, therefore we get (\ref{eq:transition}) by restricting to finitely many variables.
We see that
\begin{equation}
s_{\lambda,k}(x_0,x_1,\dots,x_{m-1})=\sum_{\mu\preceq\lambda'} K_{\lambda'\mu} f_{\mu,k}(x_0,x_1,\dots,x_{m-1}).\label{eq:basis}
\end{equation}
Now let $\lambda$ satisfy $\lambda_1+k>m$. We want to show that for all $\mu \preceq \lambda'$ we have $\ell(\mu)+k>m$. We have $\mu' \succeq \lambda$, hence $\ell(\mu)+k=\mu'_1+k\geq \lambda_1+k>m$. Therefore, all $f_{\mu,k}$ appearing in the right sum are in $J_m$, and all the $K_{\lambda'\mu}$ are integers. Therefore, $s_{\lambda,k}(x_0,x_1,\dots,x_{m-1})\in J_m$. We now finish the proof of Theorem $\ref{t:leader}$.

\begin{proof}
	Let $G_m=\{s_{\lambda,k}(x_0,x_1,\dots,x_{m-1})|\lambda_1+k>m\}$. We verify that $(S,G_m)$ is a Gr\"obner-Shirshov  basis. 	First, $G_m$ is a generating set of $J_m$. Next, by Proposition \ref{prop:basis_funs}, Equation \ref{eq:basis}, and the fact that \emph{every} monomial $x^{(\mu)}$ in $\mathcal{D}\mathbb{F}_m$ can be written as an integral combination of $s_{\mu,k}(x_0,x_1,\dots,x_{m-1})$ with $\lambda>_\text{lex}\mu$, we see that every polynomial $f\in J_m$ has a unique integral expansion in terms of $s_{\lambda,k}(x_0,x_1,\dots,x_{m-1})\in G_m$. Let $p,q\in G_m$ have a composition. Then $S(p,q)$ can be written as a linear combination of some terms in $S$ to make the monomials reduced and $s_{\mu,k}(x_0,x_1,\dots,x_{m-1})\in G_m$ where we have $w\geq_{\text{lex}}x^{(\mu)}=LT(s_{\mu,k}(x_0,x_1,\dots,x_{m-1}))$. Therefore, $S(p,q)\equiv 0 \mod (S,G_m)$. The remaining compositions are easily seen to be $ \equiv 0 \mod S$ or $(S,G_m)$, respectively. Therefore, $(S,G_m)$ is closed under compositions, hence a Gr\"obner-Shirshov pair for $W_m$.
\end{proof}
\subsection{Proof of Theorem \ref{t:revbasis}}\label{s:revbasis}

Now, fix an integer $k$ and a partition $\lambda$ such that $\ell(\lambda)+k>m$. We can see that $D_{\lambda \mu}$ are integers satisfying $D_{\lambda\lambda}=\pm 1$ and $D_{\lambda\mu}=0$ unless $\mu\succ\lambda$. Unfortunately, if $\ell(\lambda)>m/2$ then some terms will drop out of the expression for $f_{\lambda,k}$, and the remaining leading coefficient can be something other than $1$. Therefore, we only consider the characteristic 0 case in this section. In this case, notice that the notion of Gr\"obner-Shirshov bases is purely replaced by Gr\"obner bases as in Subsection \ref{Gbases}, due to Remark \ref{hyperrmk}\eqref{hyperrmk1}.

Let $S_{\text{revlex}}=\{f_{\lambda,k}(x_0,x_1,\dots,x_{m-1}) \mid 2\leq k \leq m+1, \ell(\lambda) \geq m-k+1\}$. We show that $S_{\text{revlex}}$ is a Gr\"{o}bner basis of $W_{\mathbb{F}}(m)$. We remark that $S_{\text{revlex}}$ is larger than we need to be a Gr\"{o}bner basis of $W_\mathbb F(m)$ for all $m\geq 1$. It needs to be proven that $S_{\text{revlex}}$ gives the correct dimension for the basis of $W_\mathbb F(m)$.

We first give an alternative characterization of the monomials appearing in $LM(S_{\text{revlex}})$. Let $\mu$ be a partition such that $2\leq \ell(\mu) \leq  m/2 $. Let $\eta\geq_{\text{revlex}}\mu$ be the least partition of $|\mu|$ with respect to reverse lexicographic order such that $\ell(\eta) = m-\ell(\mu)+1$. We have following:
\begin{lem}
	Let $i^*$ be the least index such that $\eta_{i^*} < \mu_{i^*}$. Then $\eta_i=1$ for all $i \in \{i^*,i^*+1,\dots, \ell(\eta)\}$.
\end{lem}
\begin{proof}
	Suppose $\eta_i > 1$ for some $i > i^*$ and that $i$ is the maximum with this property. Then we could move a box in $\eta$ from column $i$ to column $i^*$ to get a partition $\rho\vdash |\mu|$ such that $\mu\leq_\text{revlex}\rho<_\text{revlex}\eta$ and $\ell(\rho)=\ell(\eta)= m-\ell(\mu)+1$ contradicting our choice of $\eta$.
\end{proof}

The above lemma can be used to prove that $\eta$ is given by the following algorithm:

\textbf{Algorithm 1}
\begin{itemize}
	\item[] Let $r:=m-2\ell(\mu)+1$
	\item[] Let $\eta:=\mu$
	\item[] Let $i:=\ell(\mu)$
	\item[] \textbf{while} $i\geq 1$ \textbf{do}
	\item[] \qquad \textbf{if} $\mu_i \leq r$ \textbf{then}
	\item[] \qquad \qquad $r:=r-\mu_i+1$
	\item[] \qquad \qquad Append $1$ to the end of $\eta$ ($\mu_i-1$ times)
	\item[] \qquad \qquad $\eta_i:=1$
	\item[] \qquad \qquad $i:=i-1$
	\item[] \qquad \textbf{else}
	\item[] \qquad \qquad $\eta_i:=\mu_i-r$
	\item[] \qquad \qquad Append $1$ to the end of $\eta$ ($r-1$ times)
	\item[] \qquad \qquad \textbf{break}
	\item[] \qquad \textbf{end if}
	\item[] \textbf{end do} 
	\item[] \textbf{return} $\eta$
\end{itemize}

Now, let $\nu$ be the greatest partition of $|\mu|$ with respect to the reverse lexicographic order such that $\ell(\nu)=\ell(\mu)$ and $\nu\leq_{\text{revlex}}\eta$.
\begin{lem}\label{lem:index}
	Let $i^*$ be the least index such that $\eta_{i^*} < \mu_{i^*}$. Then $\mu=\nu$ if, and only if, $\mu_{i^*}-\mu_{\ell(\mu)}\leq 1$.
\end{lem}
\begin{proof}
	If $1\leq i'<i^*$ then $\eta_{i'}=\mu_{i'}$, hence $\mu_{i'}=\nu_{i'}$. Now suppose that $\mu_{i^*}-\mu_{\ell(\mu)}>1$. Then we can move one box from the $i^*$ column of $\mu$ to another non-empty column to obtain a new partition $\rho$. In the worst case $\rho_{i^*}=\eta_{i^*}$, and $\rho_{i^{*}+1} > \eta_{i^{*}+1}=1$. Therefore, $\eta>_{\text{revlex}}\rho>_{\text{revlex}}\mu$ but $\nu$ is the greatest partition of $|\mu|$ that is $\leq_{\text{revlex}} \eta$, hence $\mu\neq \nu$. Conversely, if $\mu_{i^*}-\mu_{\ell(\mu)}\leq 1$ then no such move can be made. It follows that $\mu=\nu$.
\end{proof}

Next, we give a characterization of the monomials that are unreduced with respect to the leading monomials of polynomials in $S_{\text{revlex}}$.

\begin{lem}\begin{enumerate}
		\item Let $x^{\mu}=x_0^{f_0}x_1^{f_1}\cdots x_{m-1}^{f_{m-1}}$ with $f_i=m_i(\mu)$ be a monomial of degree $k\leq m/2$ that is not in $ LM(S_{\text{revlex}})$. Then for $1\leq i \leq m-1$ we have $(i-1)f_i+i f_{i+1} \leq m-2\sum_{j=i}^{m-1} f_j$. The converse is also true.
		\item If $x^\lambda$ is unreduced and $\text{deg}(x^\lambda)>m/2$ and $\lambda^+$ is the subsequence of $\lambda$ with 0 removed, then $x^{\lambda^+}$ is in the set $\mathcal R_1.$
	\end{enumerate}
\end{lem}
\begin{proof}
	(1) Suppose it were the case that $(i-1)f_i+i f_{i+1} > m-2\sum_{j=i}^{m-1} f_j$ for some $i \in \{1,2,\dots,m-2\}$. Consider the subsequence $x_i^{f_i}x_{i+1}^{f_{i+1}}\cdots x_{m-1}^{f_{m-1}}$ of $x^{\mu}$. On $\mu$, Algorithm 1 would terminate when $\mu_{i^*}=i$ or $i+1$. In either case $\mu_{i^*}-\mu_{\ell(\mu)}\leq 1$, which would imply (by Lemma \ref{lem:index}) that $x^{\mu}$ is the leading monomial of $f_{\eta,\ell(\mu)}\in S_{\text{revlex}}$. We can see that the coefficient of $x^{\mu}$ is non-zero using the definition of $D_{\eta\mu}$. In that expression, we have $\eta=(\eta_1)\uplus(\eta_2)\uplus \cdots \uplus (\eta_{i^*})\uplus (\eta_{i^*+1},1^{\mu_{i^*+1}-\eta_{i^*+1}})\uplus(1^{\mu_{i^*+2}})\uplus\cdots \uplus (1^{\mu_{\ell(\mu)}})$. Therefore, for $x^{\mu}$ to be reduced we must have $(i-1)f_i+i f_{i+1} \leq m-2\sum_{j=i}^{m-1} f_j$ for $1\leq i \leq m-2$.
	
	Now, suppose that $(i-1)f_i+i f_{i+1} \leq m-2\sum_{j=i}^{m-1} f_j$ for all $i \in \{1,2,\dots,m-2\}$. Let $g_i\leq f_i, 0 \leq i \leq m-2$. Then
	\begin{equation*}
	(i-1)g_i+ig_{i+1}\leq (i-1)f_i+if_{i+1}\leq m-2\sum_{j=i}^{m-1} f_j\leq m-2\sum_{j=i}^{m-1} g_j.
	\end{equation*} 
	Therefore, no subsequence of $x^{\mu}$ is in $LM(S_{\text{revlex}})$, which implies that $x^{\mu}$ is unreduced in $S_{\text{revlex}}$.
	
	(2) Now suppose $\text{deg}(x^{\lambda})>m/2$. Let $\lambda^+$ be the subsequence of $\lambda$ without 0 as a part. If $\ell(\lambda^+)> m/2$ then $f_{\lambda^+,k} \in S_{\text{revlex}}$ and $LM(f_{\lambda^+,k})=x^{\lambda}$. Therefore, for $x^\lambda$ to be reduced we must have $\ell(\lambda^+)\leq m/2$. But any subsequence of a reduced monomial must be reduced, hence $x^{\lambda^+}$ is reduced and is in $\mathcal R_1.$
\end{proof}

\begin{thm}
	There are $2^m$ monomials in $\mathcal R_1 \cup \mathcal R_2$. In particular, $\mathcal R_1 \cup \mathcal R_2$ is a basis of $W(m)$.
\end{thm}

\begin{proof}
	
	We start by enumerating the monomials whose degree in $x_0$ is 0. For $t \geq 1, 0\leq \ell\leq m/2, s \geq 0$ define $g_{t,\ell,s,m}$ to be the number of monomials of degree $\ell$ in $\cal R_1$ for which each $x_i$ satisfies $t \leq i \leq m-1$, and such that the degree of $x_t$ is $s$. For reasons that will become clear, we also define $g_{t,\ell,-1,m}:=g_{t,\ell+1,0,m}$.
	
	We now prove that the function $g_{t,\ell,s,m}$ satisfies the following recursion for $0\leq \ell \leq m/2$, $1\leq t \leq m-1$, and $-1 \leq s \leq \ell$:
	\begin{equation}
	g_{t,\ell,s,m}=\begin{cases}
	\sum_{j=0}^{\ell-s} H(m-2\ell-tj-(t-1)s) g_{t+1,\ell-s,j,m}&  \text{ if }m-2\ell\geq(t-1)s,\\
	0& \text{ otherwise,}
	\end{cases}
	\label{eq:recur}
	\end{equation}
	where
	\begin{equation*}
	H(n)=\begin{cases}
	1 &\text{ if } n\geq 0\\
	0 &\text{ otherwise.}
	\end{cases}
	\end{equation*}
	The boundary conditions are:
	\begin{equation*}
	g_{m,\ell,s,m}=\begin{cases}
	1 & \text{ if }\ell=s=0,\\
	0 & \text{ otherwise}.
	\end{cases}
	\end{equation*}
	These are easily verified. If $s\geq 0$ then Equation \eqref{eq:recur} follows if we let $X$ be a monomial counted by $ g_{t,\ell,s,m}$ under the given conditions on $t,\ell,s,$ and $m$. Then $m-2\ell\leq(t-1)s$, and if we delete all the occurrences of $x_t$ from $X$ the resulting monomial satisfies $m-2\ell\geq tj+(t-1)s $, is of degree $\ell-s$, and has only terms $x_i$ where $i \geq t+1$. Conversely, if $X'$ is a monomial counted by $g_{t+1,\ell-s,j,m}$ for a $0\leq j \leq \ell-s$, and $m-2\ell\leq(t-1)s$ and $ m-2\ell\geq tj+(t-1)s $, then we can and do multiply $X'$ by $x_t^s$ to get a monomial counted by $g_{t,\ell,s,m}$. 
	
	If $s=-1$ then Equation \eqref{eq:recur} formally gives:
	\begin{align*}
	g_{t,\ell,-1,m}&=\sum_{j=0}^{\ell+1} H(m-2\ell-tj+(t-1)) g_{t+1,\ell+1,j,m}\\
	&=\sum_{j=0}^{\ell+1} H(m-2(\ell+1)-tj)g_{t+1,\ell+1,j,m}\\
	&=g_{t,\ell+1,0,m}.
	\end{align*}
	In the second line, we have used the fact that if $m-2\ell < tj-(t-1)$ then $m-2(\ell+1) < tj$, but if $m-2\ell \geq tj-(t-1)$ and $m-2(\ell+1) < tj$ then $g_{t+1,\ell+1,j,m}=0$. This explains why it makes sense to define $g_{t,\ell,-1,m}$ as we have done, and proves Equation \eqref{eq:recur} in the case $s=-1$.
	
	The recursion in \eqref{eq:recur} is inductive since each term involving $g$ the right increases in $t$. Therefore, it determines the values of $g$.
	
	Next, we prove the following relation, where $1\leq t \leq m-1, \ell\geq 1$ and $0\leq s\leq \ell$ and $m\geq 1$:
	\begin{equation}
	g_{t,\ell,s,m}=\begin{cases}
	g_{t,\ell,s+1,m-1}+g_{t,\ell-1,s-1,m-1}, \text{ if } m-2\ell\geq(t-1)s\\
	0, \text{ otherwise}.
	\end{cases}\label{eq:recur_binary}
	\end{equation}
	If $m-2\ell<(t-1)s$ then $g(t,\ell,s,m)=0$ so there is nothing to prove. Suppose $m-2\ell\geq(t-1)s$. We use reverse induction on $t$. We verify
	\begin{align*}
	g_{m-1,\ell,s,m}&=g_{m-1,\ell,s+1,m-1}+g_{m-1,\ell-1,s-1,m-1}=0,\text{ if }(\ell,s)\neq (1,1)\\
	g_{m-1,1,1,m}&=g_{m-1,1,2,m-1}+g_{m-1,0,0,m-1}=1.
	\end{align*}
	Suppose, now, that Equation \eqref{eq:recur_binary} holds for $t+1$. By \eqref{eq:recur} we have:
	
	\begin{align*}
	g_{t,\ell,s,m}&=g_{t+1,\ell-s,0,m}+\sum_{j=1}^{\ell-s} H(m-2\ell-tj-(t-1)s) g_{t+1,\ell-s,j,m}\\
	&=g_{t+1,\ell-s,0,m}+\sum_{j=1}^{\ell-s} H(m-2\ell-tj-(t-1)s) g_{t+1,\ell-s,j+1,m-1}\\
	&\qquad+\sum_{j=1}^{\ell-s} H(m-2\ell-tj-(t-1)s)g_{t+1,\ell-s-1,j-1,m-1}\\
	&=g_{t+1,\ell-s,1,m-1}+g_{t+1,\ell-s-1,-1,m-1}\\
	&\qquad+\sum_{j=1}^{\ell-s} H(m-1-2(\ell-1)-t(j+1)-(t-1)(s-1)) g_{t+1,\ell-s,j+1,m-1}\\
	&\qquad+\sum_{j=1}^{\ell-s} H(m-1-2\ell-t(j-1)-(t-1)(s+1))g_{t+1,\ell-s-1,j-1,m-1}\\
	&=g_{t+1,\ell-s,1,m-1}+g_{t+1,\ell-s,0,m-1}\\
	&\qquad+\sum_{j=2}^{\ell-s+1} H(m-1-2(\ell-1)-tj-(t-1)(s-1)) g_{t+1,\ell-s,j,m-1}\\
	&\qquad+\sum_{j=0}^{\ell-s-1} H(m-1-2\ell-tj-(t-1)(s+1))g_{t+1,\ell-s-1,j,m-1}\displaybreak[0]\\
	&=\sum_{j=0}^{\ell-s} H(m-1-2(\ell-1)-tj-(t-1)(s-1)) g_{t+1,\ell-s,j,m-1}\\
	&\qquad+\sum_{j=0}^{\ell-s-1} H(m-1-2\ell-tj-(t-1)(s+1))g_{t+1,\ell-s-1,j-1,m-1}\\
	&=g_{t,\ell-1,s-1,m-1}+g_{t,\ell,s+1,m-1},
	\end{align*}
where in the second equality, we have used the inductive hypothesis, in the third we use a simple algebraic equivalence of the terms inside $H$, in the fourth we shift the index of the sums, in the fifth we rearrange the sums and use the fact that $g(t+1,\ell-s,\ell-s+1,m-1)=0$, and in the last we used \eqref{eq:recur} again.

	We now use recursion \eqref{eq:recur_binary} to give a familiar recursion for the number of monomials in $\cal R_1$ with degree $\ell$ and $x_0$ degree 0. Let $B_{m,\ell}$ denote the cardinality of this set. Then, for $1\leq \ell\leq \lfloor m/2 \rfloor$,
	\begin{align*}
	B_{m,\ell}&=g_{1,\ell,0,m}+\sum_{s=1}^\ell g_{1,\ell,s,m}\\
	&=g_{1,\ell,1,m-1}+g_{1,\ell-1,-1,m-1}+\sum_{s=1}^\ell g_{1,\ell,s+1,m}+\sum_{s=1}^\ell g_{1,\ell-1,s-1,m}\\
	&=g_{1,\ell,1,m-1}+g_{1,\ell,0,m-1}+\sum_{s=2}^{\ell+1} g_{1,\ell,s,m}+\sum_{s=0}^{\ell-1} g_{1,\ell-1,s,m}\\
	&=B_{m-1,\ell}+B_{m-1,\ell-1}.
	\end{align*}
	When $\ell=0,$ there is just the monomial $1$ so $B_{0,m}=1$. There are $m-1$ monomials of degree 1, hence $B_{1,m}=m-1$ when $m>0$. These boundary conditions and recursions have the solution:
$
	B_{m,\ell}=\binom{m}{\ell}-\binom{m}{\ell-1}.
$
	The total number of monomials of degree $k$ satisfying (1) is hence $\sum_{\ell=0}^k B_{m,\ell}=\binom{m}{k}$. By definition, the total number of monomials of degree $k$ in $\cal R_2$ is $\binom{m}{m-k}=\binom{m}{k}.$ The total number of monomials in $\cal R_1 \cup \cal R_2$ is $	\sum_{k=0}^m \binom{m}{k}=2^m$. \end{proof}


\begin{thebibliography}{}
	
	\bibitem{AL} W. Adams and P. Loustaunau, An Introduction to Gr\"{o}bner Bases, American Mathematical Society, GSM {\bf 3} (1994).
	
	\bibitem{bm} A. Bianchi and A. Moura, {\em Finite-dimensional representations of twisted loop algebras},  Comm. Algebra {\bf 42} (2014),  3147--3182.
	
	\bibitem{BMM} A. Bianchi, A. Moura, and T. Macedo, {\em On Demazure and local Weyl modules for affine hyperalgebras}, Pacific. J. Math. {\bf 274} (2015), 257--303.
	
	\bibitem{sam} S. Chamberlin, {\em Integral bases for the universal enveloping algebras of map algebras},  J. Algebra {\bf 377}, 232--249.
	
	\bibitem{Csurvey} V. Chari, {\em Representations of Affine and Toroidal Lie Algebras}, Fields Institute Communications, in: Geometric representation theory and extended affine Lie algebras, edited by E. Neher, A. Savage, W. Wang (2011).
	
	\bibitem{CFK} V.~Chari, G.~Fourier, and T.~Khandai, {\em A categorical approach to Weyl modules}, Transf. Groups {\bf 15} (2010), 517--549.
	
	\bibitem{CFS} V. Chari, G. Fourier, and P. Senesi, {\em Weyl Modules for the twisted loop algebras},  J. Algebra {\bf 319} (2008), no. 12, 5016--5038.
	
	\bibitem{CL} V. Chari and S. Loktev, {\em Weyl, Demazure and fusion modules for the current algebra of $\lie{sl}_{r+1}$}, Adv. in Math., {\bf 207} (2006), no. 2, 928--960.
	
	\bibitem{CPweyl} V. Chari and  A. Pressley, {\em Weyl modules for classical and quantum affine algebras}, Represent. Theory  {\bf 5}  (2001), 191--223.
	
	\bibitem{CV} V. Chari and  R. Venkatesh, {\em Demazure modules, fusion products and Q-systems}, Commun. Math. Phys. {\bf  333} (2015).
	
	\bibitem{GF} G. Fourier, {\em New homogeneous ideals for current algebras: filtrations, fusion products and Pieri rules}, Mosc. Math. J. {\bf 15} (2015), 49–-72.
		
	\bibitem{FKKS} G. Fourier, T. Khandai, D. Kus, and A. Savage, {\em Local Weyl modules for equivariant map algebras with free abelian group actions}, J. Algebra {\bf 350} (2012), 386--404.
	
	\bibitem{FoLi} G. Fourier and P. Littelmann, {\em Weyl modules, Demazure modules, KR-modules, crystals, fusion products and limit constructions}, Adv. in Math. 	{\bf 211} (2007), no. 2, 566--593.
	
	\bibitem{FNS} G. Fourier, N. Manning, and A. Savage, {\em Global Weyl Modules for Equivariant Map Algebras}, Int. Math. Res. Notices {\bf t} (2015), no. 7, 1794--1847
	
	\bibitem{G} H. Garland, {\em The arithmetic theory of loop algebras}, J. Algebra {\bf 53} (1978), 480--551.
	
	\bibitem{hyperlar} D. Jakelic and A. Moura, {\em Finite-dimensional representations of hyper loop algebras}, Pacific J. Math. {\bf 233} (2007), no. 2, 371--402.

	\bibitem{JMsurvey} D. Jakelic and A. Moura, {\em On Weyl modules for quantum and hyper loop algebras}, Contemp. Math. {\bf 623} (2014), 99--134.
	
	\bibitem{KL} S.-J. Kang and K.-H. Lee, {\em Gr\"obner-Shirshov bases for representation theory},  J. Korean Math. Soc. {\bf 37} (2000), no. 1, 55--72.
	
	\bibitem{KLLO} S.-J. Kang, I.-S Lee, K.-H. Lee, and H. Oh, {\em Hecke algebras, Specht modules and Gr\"obner-Shirshov bases},  J. Algebra {\bf 252} (2002), 258--292.
	
	\bibitem{kosagz} B. Kostant, {\em Groups over $\mathbb Z$}, Algebraic Groups and Discontinuous Subgroups, Proc. Symp. Pure Math. IX, Providence, AMS (1966).
	
	\bibitem{kus} D. Kus and P. Littelmann, {\em Fusion products and toroidal algebras}, Pacific J. Math. {\bf 278} (2015), 427--445
	
	\bibitem{Mac} Macdonald, I. G., Symmetric functions and Hall polynomials. Oxford Math. Monographs (1995). The Clarendon Press Oxford University Press. 
	
	\bibitem{mitz} D. Mitzman, {\em Integral Bases for Affine Lie Algebras and Their Universal Enveloping Algebras}, Contemp. Math. {\bf 40} (1983).
	
	\bibitem{MS} T. Moro and M. Sala, {\em On the Groebner basis of some symmetric systems and their application to coding theory},  J. Symbolic Comput.  {\bf 35} (2003), no. 2, 177--194.
	
	\bibitem{NAOI} K. Naoi, {\em Weyl modules, Demazure modules and finite crystals for non-simply laced type}, Adv. in Math. {\bf 229} (2012), no. 2, 875--934.
	
	\bibitem{nehsav} E. Neher and A. Savage, {\em A survey of equivariant map algebras with open problems}, 165–182, Contemp. Math. 602, Amer. Math. Soc., Providence, RI, (2013).
	
	\bibitem{nehsav2} E. Neher and A. Savage, {\em Extensions and block decompositions for finite-dimensional representations of equivariant map algebras}, Transf. Groups {\bf 20} (2015), 183--228.
	
	\bibitem{nehsavsen} E. Neher, A. Savage, and P. Senesi, {\em Irreducible finite-dimensional representations of equivariant map algebras}, Trans. Amer. Math. Soc., {\bf 364} (2012), no. 5, 2619--2646. 

	\bibitem{stab} K. Raghavan, B. Ravinder, and S. Viswanath, {\em Stability of the Chari-Pressley-Loktev bases for local Weyl modules of $\lie{sl}_2[t]$}, Algebr. Represent. Theor. {\bf18} (2015), 613--632.
	
	
	\bibitem{stab2} K. Raghavan, B. Ravinder, and S. Viswanath, {\em On Chari-Loktev bases for local Weyl modules in type $A$}, preprint (2016).
	
	
	\bibitem{stab1} B. Ravinder, {\em Stability of the Chari-Pressley-Loktev bases for local Weyl modules of $\lie{sl}_{r+1}[t]$}, preprint (2016).
	

	\bibitem{Z} M. Zabrocki, {\em An Introduction to Symmetric Functions}, unpublished work 
	
	{\url{http://garsia.math.yorku.ca/ghana03/mainfile.pdf}}
	
	\bibitem{bw} B. Wilson, {\em Highest-weight theory for truncated current Lie algebras}, J. Algebra {\bf 336} (2011), 1--27.
	
	
	\
	
	\
	
\end{thebibliography}
\end{document}